\documentclass[a4paper]{article}

\usepackage{
amsmath,
amsthm,
amscd,
amssymb,
}
\usepackage{comment}
\usepackage{tikz}
\usetikzlibrary{positioning,arrows}
\usetikzlibrary{decorations.pathreplacing}
\usetikzlibrary{decorations.pathmorphing}
\usetikzlibrary{patterns}
\usepackage{xstring}

\usepackage{multicol}

\usepackage{ifthen}

\usepackage{subcaption}

\newtheorem{lemma}{Lemma}
\newtheorem{definition}{Definition}
\newtheorem{corollary}{Corollary}
\newtheorem{proposition}{Proposition}

\newtheorem{theorem}{Theorem}

\newtheorem{question}{Question}

\newtheorem{claim}{Claim}

\newcommand{\R}{\mathbb{R}}

\newcommand{\Z}{\mathbb{Z}}

\newcommand{\N}{\mathbb{N}}

\newcommand{\ZGC}{ZGC}

\newcommand{\supp}{\mathrm{supp}}

\tikzset{snake it/.style={decorate, decoration={snake, segment length=1mm, amplitude=0.2mm}}}

\newcommand{\leftbord}{
\draw[snake it] (0.1,0) -- (0.1,1);
}

\newcommand{\rightbord}{
\draw[snake it] (0.9,1) -- (0.9,0);
}

\newcommand{\bndr}{
\draw[color=gray](0,0) rectangle (1,1);
}

\newcommand{\tle}[1]{
    \ifthenelse{\equal{#1}{ }}{
    }{}
    \ifthenelse{\equal{#1}{n}}{
        \draw (0.3,0.25) -- (0.5,0.75) -- (0.7,0.25);
    }{}
    \ifthenelse{\equal{#1}{s}}{
        \draw (0.3,0.75) -- (0.5,0.25) -- (0.7,0.75);
    }{}
    \ifthenelse{\equal{#1}{x}}{
        \draw (0.3,0.75) -- (0.5,0.25) -- (0.7,0.75);
        \draw (0.3,0.25) -- (0.5,0.75) -- (0.7,0.25);
    }{}
}

\newcommand{\tile}[1]{
\tikz[scale=0.4,baseline=2]{
   \bndr;
   \tle{#1};
}
}

\newcommand{\ltile}[1]{
\tikz[scale=0.4,baseline=2]{
   \bndr;
   \leftbord;
   \tle{#1};
}
}

\newcommand{\rtile}[1]{
\tikz[scale=0.4,baseline=2]{
   \bndr;
   \rightbord;
   \tle{#1};
}
}

\newcommand{\lrtile}[1]{
\tikz[scale=0.4,baseline=2]{
   \bndr;
   \leftbord;
   \rightbord;
   \tle{#1};
}
}

\newcommand{\smalltile}[1]{
\tikz[scale=0.2,baseline=2]{
   \bndr;
   \tle{#1};
}
}

\newcommand{\ES}{\mathrm{ES}}
\newcommand{\spa}{\mathrm{spa}}

\newcommand{\Surf}{\mathrm{Surf}}
\newcommand{\OC}[1]{\overline{\mathcal{O}(#1)}}

\title{A Three-Dimensional SFT with Sparse Columns}

\author{
  Ville Salo
  \footnote{Author supported by Academy of Finland grant 2608073211.}
  \ \ \ \ \ \ and\ \ \ \ \
  Ilkka T\"orm\"a
  \footnote{Author supported by Academy of Finland grant 346566.}
  \\
  Department of Mathematics and Statistics \\
  University of Turku, Finland \\
  \{\texttt{vosalo}, \texttt{iatorm}\}\texttt{@utu.fi}
}

\begin{document}
\maketitle

\begin{abstract}
We construct a nontrivial three-dimensional subshift of finite type whose projective $\Z$-subdynamics, or $\Z$-trace, is 2-sparse, meaning that there are at most two nonzero symbols in any vertical column. The subshift is deterministic in the direction of the subdynamics, so it is topologically conjugate to the set of spacetime diagrams of a partial cellular automaton. We also present a variant of the subshift that is defined by Wang cubes, and one whose alphabet is binary.
\end{abstract}

\section{Introduction}

Multidimensional subshifts are shift-invariant topologically closed subsets of $A^{\Z^d}$ where $A$ is a finite alphabet. These are the objects of study in the field of symbolic dynamics. Of particular interest are the subshifts of finite type or SFTs, which are the subshifts defined by finitely many forbidden patterns. Equivalently (as far as dynamical properties go), these are the sets of valid tilings given by Wang tiles (for $d = 2$) or Wang hypercubes (for $d \geq 3$), i.e.\ they are the sets of valid ways to assign hypercubes to elements $\vec v \in \Z^d$ from a finite set of hypercubes with colors on their $(d-1)$-dimensional faces, so that matching faces have matching colors.

In this paper, we study the following problem: if a subshift of finite type in three dimensions has the property that in every valid configuration, every vertical line contains only a bounded number of ``nonzero'' symbols, can the subshift contain any nonzero point at all? It is known that in one and two dimensions, the answer is \emph{no}. In one dimension, this is next to trivial. In two dimensions, it was originally proved in \cite{PaSc15}, and generalizations are proved in \cite{Sa20,SaTo21}. We show in the present paper that in three dimensions, subshifts of finite type \emph{can} have sparse traces, by giving a concrete example of such a subshift, and the bound we obtain is two (in the general formalism of SFTs -- with Wang cubes we need four cubes per vertical line).

The set of configurations that can appear on a particular one-dimensional axis in a subshift $X$ are referred to as \emph{projective subdynamics}, or \emph{trace}, which we write $T(X)$. The trace is itself a subshift of lower dimension, but need not be of finite type even if the original subshift is. The traces and projective subdynamics of two-dimensional subshifts have been studied extensively \cite{JoKaMa07,Ho11,PaSc15,Sc15} and applied in the proofs of various results \cite{Pa13,CyKr15}. 

We say a subshift is \emph{sparse}, if there is a special \emph{zero symbol} in the alphabet, and the number of nonzero symbols is bounded by some constant. Specifically, if all elements of the subshift contain at most $k$ nonzero symbols, we say it is $k$-sparse. With this terminology, what we show in this paper is that a three-dimensional subshift of finite type can have a nontrivial ($2$-)sparse ($\Z$-)trace subshift (nontrivial meaning, having at least two points). 

Traces are also related to the dynamics of cellular automata: the set of spacetime diagrams of a $d$-dimensional cellular automaton forms a $(d+1)$-dimensional subshift, and dynamical properties like mixing and expansiveness can be formulated in terms of the trace subshift and its relation to the spacetime diagrams~\cite{Ku97b}. Traces of cellular automata have been studied further in~\cite{Le06,CeFoGu07,GuRi08,Sa12c}. In particular, if a cellular automaton has a sparse trace, then it is \emph{asymptotically nilpotent}, meaning that every initial configuration converges toward the all-zero configuration in the product topology. In~\cite{Sa12c}, it was shown that asymptotically nilpotent $d$-dimensional cellular automata are nilpotent and thus have trivial spacetime diagrams. Nevertheless, our subshift has the property of vertical determinism, meaning the contents of every horizontal slice determines the slice below uniquely. Thus, it can be thought of as a \emph{partial} cellular automaton which is asymptotically nilpotent, but not nilpotent.


We give some notation to be able to give more fine-grained information about the nature of sparseness of the trace. Consider a one-dimensional subshift $X$.
The \emph{sparseness} $\spa(X)$ of $X$ is the minimal $k$ such that $X$ is $k$-sparse.
Its \emph{essential sparseness} $\ES(X)$ is the minimal $k$ such that for some $r \in \N$, the positions of all nonzero symbols in any configuration $x \in X$ can be covered by a set of the form $N + \{-r, \ldots, r\}$ where $|N| \leq k$.
Equivalently, $\ES(X)$ is the minimal $k$ such that some subshift conjugate to $X$ is $k$-sparse. Write $X \cong Y$ if two subshifts (of the same dimension) are conjugate. Define
\begin{align*}
  \alpha(X) & {} = \spa(T(X)) & \beta(X) & {} = \ES(T(X)) \\
  \underline \alpha(X) & {} = \inf \{ \alpha(Y) \;|\; Y \cong X \} & \underline \beta(X) & {} = \inf \{ \beta(Y) \;|\; Y \cong X \} \\
  \overline \alpha(X) & {} = \sup \{ \alpha(Y) \;|\; Y \cong X \} & \overline \beta(X) & {} = \sup \{ \beta(Y) \;|\; Y \cong X \}
\end{align*}
Of these notions, $\alpha$ and $\beta$ are not conjugacy-invariant, and the others are.
We include $\overline\alpha(X)$ for completeness, even though $\overline \alpha(X) = \infty$ whenever $X$ contains at least one nonzero point.
In terms of these invariants, the following is our main theorem.

\begin{theorem}
\label{thm:Main}
There exists a set of Wang cubes $C$ such that the set of valid tilings $X \subset C^{\Z^3}$ is vertically deterministic and satisfies $\alpha(X) = 4$, $\underline \alpha(X) = 2$, $\beta(X) = 2$, $\underline \beta(X) = 2$, and $\overline \beta(X) = 4$.
\end{theorem}

From $\alpha(X) = 4$ we get the following.

\begin{corollary}
There exists a deterministic set of Wang cubes $C$, one of which is called the zero cube, such that $C$ admits a tiling containing a nonzero cube, but in every valid tiling every column contains at most $4$ nonzero cubes.
\end{corollary}

From $\underline \alpha(X) = 2$ we immediately get the following.

\begin{corollary}
There exists a nontrivial deterministic three-dimensional zero-pointed subshift of finite type 
such that every column of every valid configuration contains at most two non-zero symbols.
\end{corollary}

By studying the construction, we can say a bit more.
A result of Pavlov and Schraudner \cite{PaSc15} states that a $\Z$-sofic shift is the projective subdynamics of a $\Z^2$-SFT if and only if it has no \emph{universal period}.
We obtain a concrete example of a three-dimensional SFT whose projective subdynamics is a countable sofic shift with a universal period.
Define $X_{\leq k} \subset \{0,1\}^\Z$ as the subshift containing those points where the symbol $1$ appears at most $k$ times.

\begin{theorem}
There exists a three-dimensional subshift of finite type $X \subset \{0,1\}^{\Z^3}$ such that the $\Z$-projective subdynamics is $X_{\leq 2}$.
\end{theorem}

\begin{corollary}
There exists a three-dimensional subshift of finite type whose trace is not the trace of any two-dimensional subshift of finite type. Furthermore, we can take the separating trace to be sofic.
\end{corollary}

A $d$-dimensional subshift is \emph{block-gluing}, if any two hyperrectangular patterns the appear in valid tilings also occur in a single valid tiling as long as their distance is greater than a fixed constant.
On any block-gluing two-dimensional SFT, cellular automata that are asymptotically nilpotent are nilpotent~\cite{SaTo21}; we do not know whether this is the case also on block-gluing three-dimensional SFTs.

In \cite{Sa12c}, it is asked whether a two-dimensional SFT can admit a cellular automaton which is asymptotically nilpotent and not nilpotent.
In \cite{Sa20} it is shown that a two-dimensional SFT cannot admit a shift which is strictly asymptotically nilpotent (this means asymptotically nilpotent but not nilpotent; it is equivalent to having a trace that is sparse in either direction). Our example shows that in three dimensions a shift can be strictly asymptotically nilpotent:

\begin{corollary}
There exists a three-dimensional subshift of finite type $X \subset \{0,1\}^{\Z^3}$ admitting a cellular automaton $f : X \to X$ such that $f$ and $f^{-1}$ are both asymptotically nilpotent but neither of them is nilpotent.
We can pick $f$ to be a shift map.
\end{corollary}

Our constructions answer a question of Pavlov, who asked in private communication whether three-dimensional SFTs can have sparse traces, a natural question in light of \cite{PaSc15}. The question was also stated explicitly as Question~2 and Question~5 in \cite{Sa17}.

\section{Preliminaries}

\subsection{Symbolic dynamics}

Let $\Sigma$ be a finite alphabet and $G$ a group with a left invariant metric (in this article, usually $\Z^d$ for some $d \geq 1$ and the $\ell_\infty$ metric).
We assume that $\Sigma$ always contains a special symbol $0 = 0_\Sigma$ called \emph{zero}.
The \emph{full $G$-shift over $\Sigma$} is the set $\Sigma^G$ of \emph{configurations}, or colorings of $G$ using $\Sigma$.
The \emph{support} of $x \in \Sigma^G$ is the set $\supp(x) = \{ g \in G \;|\; x_g \neq 0 \}$.
We endow $\Sigma^G$ with the product topology induced by the discrete topology on $\Sigma$.
Then $G$ acts on $\Sigma^G$ from the left by homeomorphisms, which are the \emph{left shifts} $x \mapsto g x$ defined by $(g x)_h = x_{g^{-1} h}$.
A \emph{$G$-subshift} is a topologically closed and $G$-invariant set $X \subset \Sigma^G$.
A \emph{block code} is a function $\phi : X \to Y$ between subshifts $X \subset \Sigma^G, Y \subset \Gamma^G$ that is continuous and $G$-equivariant ($f(g x) = g f(x)$ for all $x \in X, g \in G$).
If $\phi$ is bijective, it is called a \emph{topological conjugacy}, and $X$ and $Y$ are \emph{topologically conjugate}.
We say $\phi$ is \emph{$0$-to-$0$} if $0_\Gamma^G \in Y$ and $\phi^{-1}(0_\Gamma^G) = \{0_\Sigma^G\}$.

Let $H \leq G$ be a subgroup and $X \subset \Sigma^G$ a subshift.
The \emph{$H$-projective subdynamics} or \emph{$H$-trace} of $X$ is the $H$-subshift $T_H(X) = \{ x|_H \;|\; x \in X \}$.
If $G = \Z^d$, the $\Z$-projective subdynamics is understood with respect to the subgroup $\{0\}^{d-1} \times \Z \cong \Z$ and denoted simply $T(X)$.

Subshifts and block codes have finitary characterizations.
A \emph{pattern} is a function $P \in \Sigma^D$ for $D = D(P) \subset G$ finite.
Every set of patterns $\mathcal{P}$ defines a subshift as the set $X_{\mathcal{P}} = \{ x \in \Sigma^G \;|\; \forall P \in \mathcal{P}, g \in G : (g x)|_{D(P)} \neq P \}$ of those configurations where none of them occur.
Every subshift is defined by a set of forbidden patterns in this way.
If $\mathcal{P}$ is finite, then $X_{\mathcal{P}}$ is a \emph{shift of finite type (SFT)}.
The image of an SFT $X$ under a block map $\phi$ is called a \emph{sofic shift}, and $\phi$ (or sometimes $X$) is its \emph{SFT cover}.
We say a subshift is \emph{$0$-to-$0$ sofic} if it is a sofic shift and has a $0$-to-$0$ SFT cover.

A subshift $X \subset \Sigma^{\Z^d}$ is \emph{deterministic} in a direction $\vec v \in \Z^d$ if whenever $x, y \in X$ satisfy $x|_H = y|_H$, where $H = \{ \vec u \in \Z^d \;|\; \vec u \cdot \vec v < 0 \}$, then $x = y$.
This is equivalent to the existence of a finite set $D \subset H$ with $x|_D = y|_D$ always implying $x_{\vec 0} = y_{\vec 0}$.

Every block map $\phi : X \to Y$ between subshifts $X \subset \Sigma^G, Y \subset \Gamma^G$ is defined by a \emph{local rule} $\Phi : \Sigma^N \to \Gamma$, where $N \subset G$ is a finite \emph{neighborhood}, as $\phi(x)_g = \Phi((g^{-1} x)|_N)$.
If $N \subset B_r(e_G)$, where $B_r(g)$ is the radius-$r$ ball around $g$ in the chosen metric on $G$, then $r$ is a \emph{radius} of $\phi$.
If $X = Y$, then $\phi$ is a \emph{cellular automaton (CA)}.
The \emph{spacetime subshift} of $\phi$ is the set of configurations $x \in \Sigma^{G \times \Z}$ where $x|_{G \times \{n+1\}} = \phi(x|_{G \times \{n\}})$ for all $n \in \Z$, and its \emph{trace} is the $\Z$-trace of this SFT.

A \emph{partial CA} is a partial function $\phi : X \to X$ defined by a local rule $\Phi : \mathcal{P} \to \Gamma$ with $\mathcal{P} \subset \Sigma^N$ similarly to a CA as $\phi(x)_g = \Phi((g^{-1} x)|_N)$, which is undefined if $(g^{-1} x)|_N \notin \mathcal{P}$ for any $g \in G$.
Its spacetime subshift and trace are defined analogously to that of a CA.
Up to topological conjugacy, $\Z^{d+1}$-SFTs that are deterministic in the axial direction $\vec e_{d+1}$ are equivalent to spacetime subshifts of partial $\Z^d$-cellular automata.

Let $C$ be a finite set of colors and $G$ a group with finite and symmetric generating set $S$.
A \emph{Wang tile} on $G$ is an $|S|$-tuple $c \in C^S$.
A \emph{tiling} over a set $W \subset C^S$ of tiles is a configuration $x \in W^G$ such that $(x_g)_s = (x_{g s})_{s^{-1}}$ holds for all $g \in G, s \in S$.
Intuitively, a Wang tile colors the edges adjacent to a vertex of the right Cayley graph of $G$, and the colorings at adjacent vertices are required to match.
The set of tilings over $W$ forms a $G$-SFT.
On $\Z^d$ we use the generating set $\{ \pm \vec e_1, \pm \vec e_2, \ldots, \pm \vec e_d \}$ and depict Wang tiles as $d$-dimensional unit hypercubes with colored faces.
In particular for $d = 3$, a Wang cube is depicted as a unit cube.

\subsection{Substitutions}

Consider now $G = \Z^d$.
A \emph{substitution on $\Sigma$ of constant shape $n_1 \times \cdots \times n_d$} is a function $\tau : \Sigma \to \Sigma^R$ where $R = \prod_{i=1}^d [0, n_i-1]$.
We extend $\tau$ to all patterns $P \in \Sigma^S$ by $\tau(P)_{\vec u} = \tau(P_{\vec v})_{\vec w}$ with $\vec u_i = n_i \vec v_i + \vec w_i$ and $0 \leq \vec w_i < n_i$ for each $1 \leq i \leq d$.
In particular, $\tau^n(s)$ is a hypercube pattern for each $s \in \Sigma$ and $n \geq 0$, and we call it an \emph{$n$-tile} of $\tau$.
A \emph{macrotile} is an $n$-tile for some $n \geq 1$.
The $\tau$-image $\tau(x)$ of a configuration $x \in \Sigma^{\Z^d}$ is defined similarly.
The subshift $X_\tau$ is defined as $X_{\mathcal{P}}$, where the forbidden set $\mathcal{P}$ contains all patterns that do not occur in any $\tau^n(s)$ for $s \in \Sigma$ and $n \geq 0$.

\subsection{Conventions and the idea of the construction}

We fix some conventions to make our construction easier to follow.
We think of $\Z^3$ as a Cartesian product of $\Z^2$ and $\Z$ where we have a stack of $\Z^2$-shaped horizontal planes on top of each other, and $\Z$ is a `vertical' axis.
Movement on $\Z^2$ is discussed with cardinal directions -- north, south, west, east -- and up and down refer to $\Z$-movement.
The same conventions hold in $\R^3$.
By default, we look at configurations from above, and north is up and west is left on the paper; we indicate the axes when deviating from this convention.
The maps
\[ \pi_H : \R^3 \to \R^2, \pi_V : \R^3 \to \R \]
are the horizontal and vertical projections, i.e. projection to the first two coordinates and projection to the last coordinate.
The word \emph{projection} by default refers to the horizontal projection.

The basis of the idea is that we would like every tiling to either be empty, or to contain one approximately horizontal surface, as then every column is 1-sparse, containing only one cell of the surface.
If the surface is flat, then we are able to place an infinite number of translated copies of it on top of itself, as the local rules of the SFT cannot see over the ``empty space'' or zero-symbols, and this prevents the SFT being sparse.
To avoid this, we make the surface fluctuate: it should not be restricted to any finite set of horizontal planes, and should explore arbitrary heights and depths.
In fact, it is not hard to see that a surface suspended in empty space can either be placed on top of itself without violating local rules, or it contains pairs of points that are horizontally at a small (bounded) distance, and vertically at an arbitrarily large distance.


We define a substitution such that every configuration in the subshift it generates can be interpreted as a surface containing horizontally close but vertically distant pairs, by recursively building bridges.
One surface with this property (interpreted as an actual surface in $\R^3$) is rendered in Figure~\ref{fig:Rendering}.
We give an outline of the proof that the subshift depicted in the figure is sparse in Section~\ref{sec:SimplerExample}, but also show that it is not $3$-sparse.
The $2$-sparse example we construct is similar in spirit, but has three bridges instead of one. Its bridges are much larger, and its rendering no longer looks very interesting.

\begin{figure}[htp]
  \begin{center}
    \includegraphics[scale=0.375]{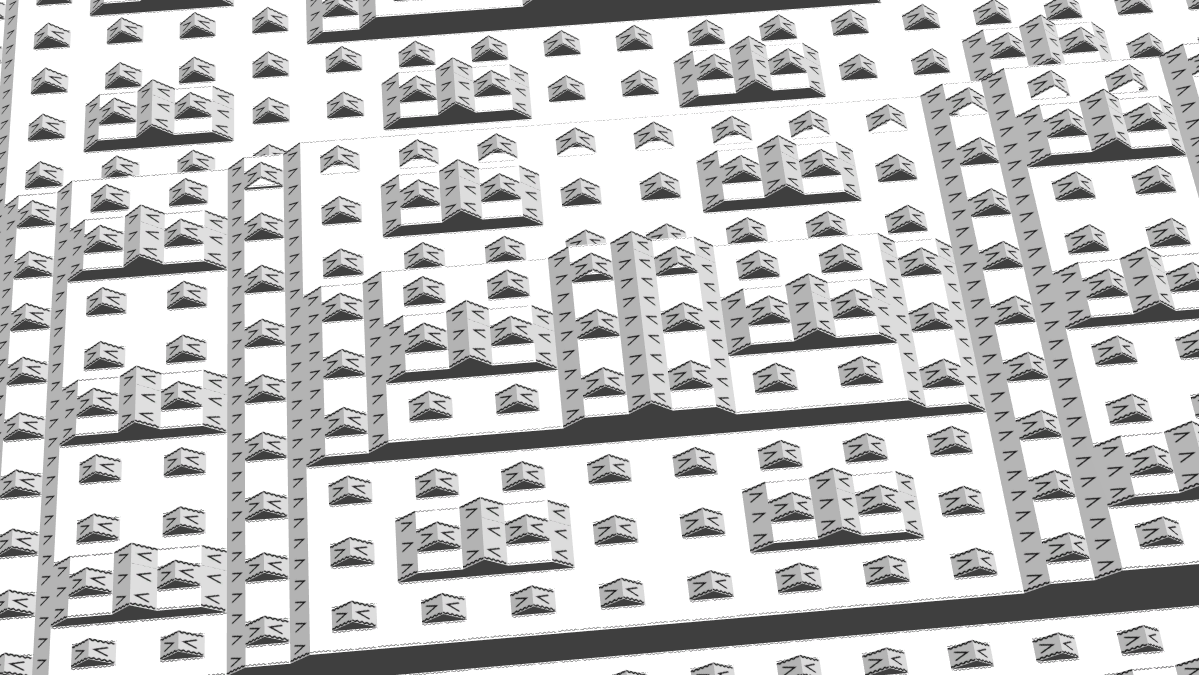}
  \end{center}
  \caption{A rendering of a fluctuating surface.}
  \label{fig:Rendering}
\end{figure}


\section{Surfaces in $\R^3$}
\label{sec:Mats}

We describe a piecewise linear surface of $\R^3$ (with some bifurcation points) such that three copies of it cannot be overlaid even if we are allowed to make small perturbations to the heights of the surface, and are allowed to translate it.
Here the relevant definition of ``overlaid'' is that the three copies almost touch at a certain horizontal coordinate.
The allowed perturbations have a rather technical definition, so we begin by defining a general class of surfaces in which they must stay.
We call them ``mats'', as it is a nice short word, but keep in mind that, typically, \emph{a mat is not flat}.

\begin{definition}
  A \emph{mat} is a nonempty compact set $M \subset \R^3$ such that the following conditions hold for some $a, b, c, d, x, y \in \R$.
  \begin{itemize}
  \item
    $\pi_H(M) = [a,b] \times [c,d]$.
  \item
    For each $(s,t) \in [a,b] \times [c,d]$, we have $|\pi_H^{-1}(s,t) \cap M| \in \{1,2\}$.
  \item
    For each $s \in [a,b]$ we have $\pi_H^{-1}(s,c) \cap M = \{x\}$ and $\pi_V^{-1}(s,d) \cap M = \{y\}$.
  \item
    Define the \emph{upper and lower functions} of $M$ as $u_M(s,t) = \max(\pi_H^{-1}(s,t) \cap M)$ and $\ell_M(s,t) = \min(\pi_H^{-1}(s,t) \cap M)$.
    For each $s \in [a,b]$, these functions are continuous on $\{s\} \times [c,d]$.
  \end{itemize}
\end{definition}

Intuitively, a mat $M$ contains the segments $[a,b] \times \{c\} \times \{x\}$ and $[a,b] \times \{d\} \times \{y\}$, which are furthermore connected by two sets of curves with constant x-coordinates.
The name ``mat'' is evocative of a weaved fabric, consisting of parallel warp strings (and orthogonal weft strings, which we ignore in this metaphor) suspended between two rigid rods in a loom.

We define a metric (allowing infinite distances) on the space of all mats.

\begin{definition}
Let $M, N \subset \R^3$ be mats.
If $\pi_H(M) \neq \pi_H(N)$, then $d(M, N) = \infty$. Otherwise, $d(M, N) = \sup_{(s,t) \in \pi_H(M)} d_H(\pi_H^{-1}(s,t) \cap M, \pi_H^{-1}(s,t) \cap N)$, where $d_H$ is the Hausdorff metric on compact subsets of $\R^3$.
\end{definition}

In words, the distance of two mats with distinct $\pi_H$-projections is infinite; and the distance between two mats with equal $\pi_H$-projections is the supremum of the Hausdorff distances between their intersections with any vertical column.
This means that two mats with the same $\pi_H$-projection are $\epsilon$-close if each point of the first mat is $\epsilon$-close to some point of the second mat on the same column, and vice versa.

\begin{lemma}
  \label{lem:MatNeighbor}
  Let $M, N \subset \R^3$ be two mats with $d(M,N) < \infty$. Then $|u_M(s,t) - u_N(s,t)| \leq d(M,N)$ and $|\ell_M(s,t) - \ell_N(s,t)| \leq d(M, N)$ for all $(s,t) \in \pi_H(M)$.
\end{lemma}

\begin{proof}
  Suppose that the claim is false.
  By symmetry, we may assume $\ell_M(s,t) > \ell_N(s,t) + d(M, N)$ for some $(s,t) \in \pi_H(M)$.
  We must then have $|u_M(s,t) - \ell_N(s,t)| \leq d(M, N)$.
  This implies $u_M(s,t) < \ell_N(s,t) + d(M, N) < \ell_M(s,t)$, a contradiction.
\end{proof}

We now prove the fundamental lemma of mats, which states that if horizontal projections of two disjoint mats overlap, then one is everywhere on top of the other.
We will use it in the reverse direction to produce intersections from collections of overlapping mats.

\begin{lemma}[Fundamental lemma of mats]
\label{lem:Fundamental}
Let $M, N \subset \R^3$ be disjoint mats such that $\pi_H(M)$ is a translate of $\pi_H(N)$. If $(x_i, y_i, z_i) \in M$ and $(x_i, y_i, z'_i) \in N$ for $i = 1, 2$, and $z_1 < z'_1$, then also $z_2 < z'_2$.
\end{lemma}

\begin{proof}
  Let $\pi_H(M) = [0,a] \times [0,b]$ and $\pi_H(N) = (a',b') + ([0,a] \times [0,b])$.
  By symmetry we may assume $a', b' \leq 0$.
  The set $L = \pi_H^{-1}([0,a] \times \{0\}) \cap M$ is a line segment of the form $[0,a] \times \{0\} \times \{x\}$.
  From the assumptions we have $0 \leq y_1 \leq b+b'$ and $\ell_M(x_1, y_1) \leq z_1 < z'_1 \leq u_N(x_1, y_1)$.
  Since both functions are continuous on $\{x_1\} \times [0, b+b']$ and the mats are disjoint, by the intermediate value theorem we must have $\ell_M(x_1, t) < u_N(x_1, t)$ for all $t \in [0, b+b']$.
  In particular $x = \ell_M(x_1, 0) < u_N(x_1, 0)$.
  
  Suppose for a contradiction that $z_2 > z'_2$.
  Then we have $x = u_M(x_2, 0) > \ell_N(x_2, 0)$ analogously to the above.

  Since mats are compact, the subsets $A_1 = \{s \in [0,a+a'] \mid u_N(s, 0) \geq x\}$ and $A_2 = \{s \in [0,a+a'] \mid \ell_N(s,0) \leq x\}$ are compact as well.
  Both are nonempty since $x_1 \in A_1$ and $x_2 \in A_2$, and because their union $[0, a+a']$ is connected, they must intersect in at least one point $x_3$.
  Now we have $\ell_N(x_3, 0) \leq x = \ell_M(x_3, 0) \leq u_N(x_3, 0)$.
  All three functions are continuous on $\{x_3\} \times [0, b+b']$ and $\ell_N(x_3, b+b') = u_N(x_3, b+b')$, so by the intermediate value theorem there must exist $t \in [0, b+b']$ with either $\ell_N(x_3, t) = \ell_M(x_3, t)$ or $\ell_M(x_3, t) = u_N(x_3, t)$.
  This contradicts the disjointness of $M$ and $N$.
\end{proof}

For mats $M, N$ such that $M$ is below $N$ in the sense of the previous lemma, we write $M < N$.
Note that this does not define a partial order on the set of all mats.

Pick numbers $a, b, c, d \in \N$ (which we will specify later).
The following diagram describes a mat $T \subset \R^3$ of particular interest to us:

\[\begin{smallmatrix}
\tile{ } & \tile{ } & \tile{ } & \tile{ } & \tile{ } & \mbox{ $\}$ $a$ times} \\
\rtile{ } & \lrtile{s} & \lrtile{n} & \lrtile{s} & \ltile{ } & \mbox{ $\}$ $b$ times} \\
\rtile{ } & \lrtile{s} & \lrtile{ } & \lrtile{s} & \ltile{ } & \mbox{ $\}$ $c$ times} \\
\rtile{ } & \lrtile{ } & \lrtile{ } & \lrtile{ } & \ltile{ } & \mbox{ $\}$ $d$ times} \\
\rtile{ } & \lrtile{n} & \lrtile{ } & \lrtile{n} & \ltile{ } & \mbox{ $\}$ $c$ times} \\
\rtile{ } & \lrtile{n} & \lrtile{s} & \lrtile{n} & \ltile{ } & \mbox{ $\}$ $b$ times} \\
\tile{ } & \tile{ } & \tile{ } & \tile{ } & \tile{ } & \mbox{ $\}$ $a$ times} \\
\end{smallmatrix}\]

Denote $e = 2(a+b+c)+d$. Each small tile represents a subset of $\R^3$ whose horizontal projection is some unit square $[i,i+1] \times [j,j+1]$ for $(i,j) \in \Z^2$. Placing the southwest corner of the southwestmost tile at the origin of $\R^3$, the tiles are placed so that their horizontal projections cover exactly the area $[0,5] \times [0, e]$. The heights come from the fact that tiles are slanted. To fix an orientation, we choose that a northward $\wedge$ implies an upward slope of $1$ when moving to the north, and similarly for a southward $\vee$. This defines a set which is not simply connected (the jagged edges of tiles imply that the sides are not at the same height).

We highlight the following parts of $T$:
\begin{itemize}
\item The function $\ell_T$ restricted to $([0,1] \cap [4,5]) \times [0, e]$ is the constant $0$ function.
  Its graph forms the \emph{west and east unbridges}, two simply connected submats with projections $[0,1] \times [0,e]$ and $[4,5] \times [0,e]$ respectively.
  \item The function $u_T$ restricted to $([1,2] \cup [3,4]) \times [0,e]$ is piecewise linear, given by $\ell_T(s, t) = \max(0, \min(b+c, t-a, a+2(b+c)+d-t))$.
    Its graph forms the the \emph{west and east up bridges}, two simply connected submats with projections $[1,2] \times [0,e]$ and $[3,4] \times [0,e]$ respectively.
  \item The function $\ell_T$ restricted to $[2,3] \times [0,e]$ is piecewise linear, given by $u_T(s,t) = \min(0, \max(-b, a-t, t-a-2(b+c)-d))$.
    Its graph forms the \emph{central down bridge}, a simply connected submat with projection $[2, 3] \times [0, e]$.
  \item
    A \emph{blade} is the graph of the restriction of $\ell_T$ or $u_T$ to $\{s\} \times [0,e]$ with $s \in \Z$.
\end{itemize}
The east and west up bridges consist of strings that, going from south to north, first ascend and then descend, and at their highest point they have height $h = b+c$. The central down bridge is, at its lowest point, at height $-b$. The unbridges are flat.
Blades are exactly the east and west borders of these structures, and they will play an important role in the coming proofs.

Let also $T^3 = \{(0,-e,0), (0,0,0), (0,e,0)\} + T$, so $T^3$ is a mat with horizontal projection $[0,5] \times [-e,2e]$. 

Choose a constant $\epsilon > 0$.
We will determine its value later in the construction, after we have obtained a collection of inequalities it and the values $a, b, c, d$ should satisfy.
Due to Lemma~\ref{lem:MatNeighbor}, whenever a mat $M$ is $\epsilon$-close to $T$ in the mat metric, it makes sense to talk about its bridges and blades just like with $T$ (although the central down bridge should also include be the union of the graph of $u_M$ restricted to $(2,3) \times [0,e]$, and analogously for the other parts: the interiors of the bridges of $M$ may contain bifurcations).

We now prove that if three $\epsilon$-perturbed translates of $T^3$ intersect a single column, and the intersection points are all $\epsilon$-close to each other, then at least two of the translates intersect.
It is a good idea to think of $\epsilon$ as an infinitesimally small number, but we need to keep track of ``how infinitesimal'' it needs to be in the proof. In practice we encounter the inequalities $4\epsilon < h$, $7 \epsilon \leq h$, $h \geq 5\epsilon$, $b > 5\epsilon$, and $\epsilon < c/4$, which follow from the assumption $\epsilon < \min(b/7, c/8)$. One should also visualize the tops and bottoms of the up and down bridges as taking up most of the tile, and $d \geq \frac{2e}{3}$ suffices in practice.


\begin{lemma}
  \label{lem:T3Lemma}
  Suppose $\epsilon < \min(b/7, c/8)$ and $d \geq \frac{2e}{3}$. Let $A_1, A_2, A_3 \subset \R^3$ be three mats in the $\epsilon$-neighborhood of $T^3$. Then there do not exist $\vec{v}_1, \vec{v}_2, \vec{v}_3 \in [0,5] \times [0,e] \times \R$ and $r_1, r_2, r_3 \in \R$ such that for all $i \neq j \in \{1,2,3\}$,
  \begin{itemize}
  \item $|r_i - r_j| < \epsilon$,
  \item $(0,0,r_i) \in A_i - \vec{v}_i$,
  \item the sets $A_i - \vec{v}_i$ and $A_j - \vec{v}_j$ are disjoint.
  \end{itemize}
\end{lemma}

\begin{proof}
  We assume that such $\vec{v}_i$ and $r_i$ can be found and derive a contradiction.
  Denote $B_i = A_i - \vec{v}_i$ and $(x_i, y_i, z_i) = \vec{v}_i$ for all $i \in \{1,2,3\}$. 

  The steps for the proof are the following.
  We show that there are $i, j \in \{1,2,3\}$ such that $|x_i - x_j| \in (4, 5]$, i.e. the opposite-side unbridges of the mats $B_i$ and $B_j$ are on top of each other.
  The situation is then as in Figure~\ref{fig:LeftRightIntersection}.
  Then, we show that the third mat $B_k$ must be below both $B_i$ and $B_j$, by having one of its up bridges close to the borders of the others.
  The situation is then as in Figure~\ref{fig:ThirdUnder}, and we see that the other up bridge of $B_k$ then intersects the down bridge of either $B_i$ or $B_j$.

  \newcommand{\TopTileAt}[2]
  {
    \draw[fill = white] (#1,#2) rectangle (#1+5,#2+7);
    \draw(#1,#2) -- (#1+5,#2);
    \draw(#1,#2+7) -- (#1+5,#2+7);
    \draw(#1+1,#2+1) rectangle (#1+2,#2+6);
    \draw(#1+2,#2+1) rectangle (#1+3,#2+6);
    \draw(#1+3,#2+1) rectangle (#1+4,#2+6);
    \node()at(#1+1.5,#2+3.5) {U};
    \node()at(#1+2.5,#2+3.5) {D};
    \node()at(#1+3.5,#2+3.5) {U};
  }

  \newcommand{\SideTileAt}[2]
  {
    \draw (#1,#2)--(#1+5,#2);
    \draw (#1+1,#2) rectangle (#1+2,#2+5);
    \draw (#1+2,#2) rectangle (#1+3,#2-3);
    \draw (#1+3,#2) rectangle (#1+4,#2+5);
    \node()at(#1+1.5,#2+2.5) {U};
    \node()at(#1+2.5,#2-1.5) {D};
    \node()at(#1+3.5,#2+2.5) {U};
  }

  \begin{figure}
    \begin{subfigure}{.5\textwidth}
      \centering
      \begin{tikzpicture}[scale=0.4]

        \TopTileAt{0}{-7}
        \TopTileAt{0}{0}
        \TopTileAt{0}{7}
        \node[draw,circle] (a) at (-2,7) {$B_i$};
        \draw[->] (a) -- (-0.3,6);

        \TopTileAt{0+4.3}{-7+2.1}
        \TopTileAt{0+4.3}{0+2.1}
        \TopTileAt{0+4.3}{7+2.1}
        \node[draw,circle] (b) at (11.7,7) {$B_j$};
        \draw[->] (b) -- (9.7,6);

        \draw[very thick,dashed] (0,14)--(0,17.8);
        \draw[very thick,dashed] (4.3,16.1)--(4.3,17.8);
        \draw [decorate,decoration={brace,amplitude=4pt},xshift=0pt,yshift=4pt]
        (0,17.8) -- (4.3,17.8) node [black,midway,yshift=10pt] {\footnotesize
          $> 4$};
        \draw[very thick,dashed] (0,0)--(-2,0);
        \draw[very thick,dashed] (4.3,2.1)--(-2,2.1);
        \draw [decorate,decoration={brace,amplitude=4pt},xshift=-4pt,yshift=0pt]
        (-2,0) -- (-2,2.1) node [black,midway,xshift=-14pt] {\footnotesize
          $\leq e$};



        \node[draw,fill,circle,inner sep=1pt] () at (4.8,4) {};

      \end{tikzpicture}
      \caption{Top view.}
      \label{fig:LeftRightIntersection}
    \end{subfigure}
    \begin{subfigure}{.5\textwidth}
      \centering
      \begin{tikzpicture}[scale=0.6]
        \SideTileAt{0}{0}
        \SideTileAt{4.3}{1.2}
        \SideTileAt{3.7}{-5.8}

        \node[draw,circle] (a) at (-0.6,3) {$B_i$};
        \draw[->] (a) -- (0.9,2.5);
        \node[draw,circle] (b) at (10,3) {$B_j$};
        \draw[->] (b) -- (8.4,2.5);
        \node[draw,circle] (c) at (8.5,-7.5) {$B_k$};
        \draw[->] (c) -- (6.9,-7);
        \draw [decorate,decoration={brace,amplitude=2pt},xshift=0pt,yshift=0pt]
        (5,1.2-0.1) -- (5,0.1) node [black,midway,xshift=12pt] {\footnotesize
          $< \epsilon$};
        \draw [decorate,decoration={brace,amplitude=2pt},xshift=0pt,yshift=0pt]
        (5,-0.1) -- (5,-0.7) node [black,midway,xshift=12pt] {\footnotesize
          $< \epsilon$};

        \draw[pattern=north west lines, pattern color=gray] (6.7,-1.8) rectangle (7.3, -0.8);
        '
        \node[draw,fill,circle,inner sep=1pt] () at (4.8,-0.2) {};

      \end{tikzpicture}
      \caption{Side view.}
      \label{fig:ThirdUnder}
    \end{subfigure}
    \caption{Typical positions of $B_i$, $B_j$, and $B_k$, viewed from $(0,0,\infty)$ and $(0,-\infty,0)$. The black dot marks the same possible origin in both figures, and the hatched area marks the intersection in the rightmost figure, which gives the final contradiction. Not drawn to scale, and the $\epsilon$-perturbation of the mats is not shown.}
  \end{figure}
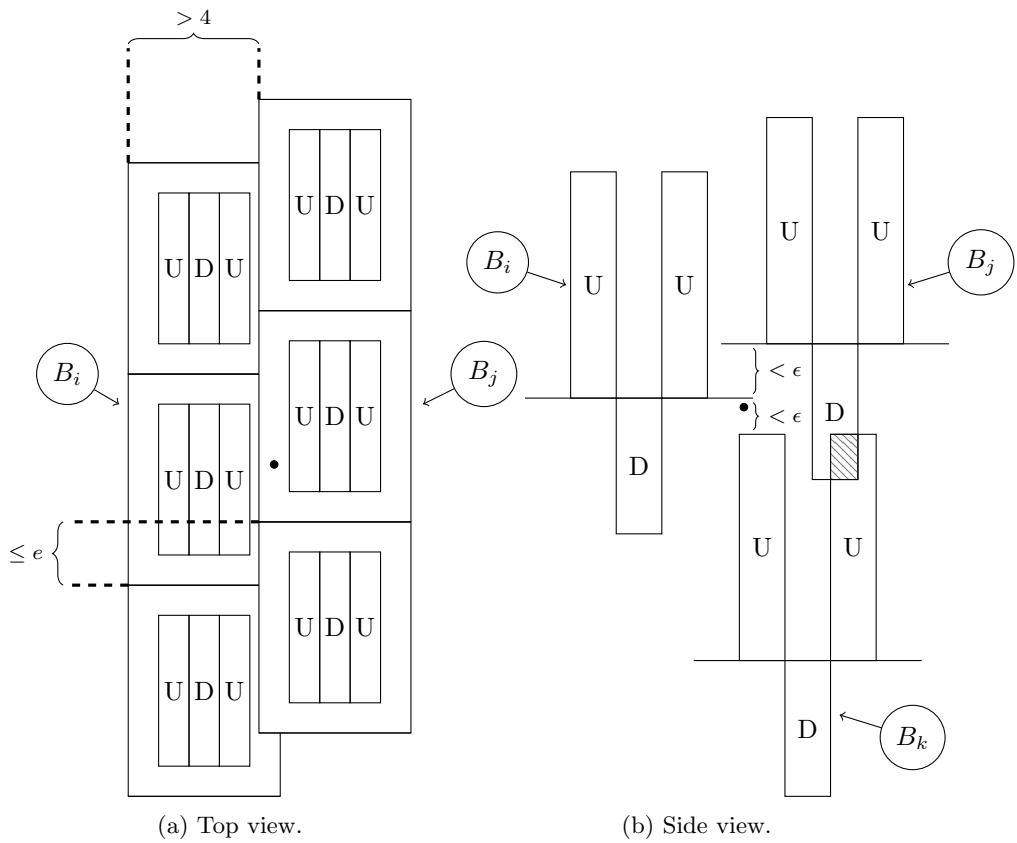

  We note that by Lemma~\ref{lem:Fundamental}, for all $i \neq j$ we have either $B_i < B_j$ or $B_i > B_j$, since the horizontal projection of each mat contains the origin and we assumed the mats to be disjoint.

  \begin{claim}
    \label{cl:xDistance}
    If $|x_i - x_j| \in [0,4]$ then $|z_i - z_j| \geq h-2\epsilon$.
  \end{claim}

  \begin{proof}
    By symmetry, we may assume $B_i < B_j$ and $x_j < x_i$, i.e. $B_i$ is below $B_j$ and to the west of it. Since $x_i - x_j \in [0,4]$, and since all bridges have the same west-east width $1$, the horizontal projection of the westmost border unbridge of $B_j$ intersects that of one of the up bridges of $B_i$:
    \begin{itemize}
    \item if $x_i - x_j \in [0, 1]$, then the projection of the west blade of the west unbridge of $B_j$ intersects the projection of the west up bridge of $B_i$,
    \item if $x_i - x_j \in [1, 2]$, then the projection of the east blade of the west unbridge of $B_j$ intersects the projection of the west up bridge of $B_i$,
    \item if $x_i - x_j \in [2, 3]$, then the projection of the west blade of the west unbridge of $B_j$ intersects the projection of the east up bridge of $B_i$,
    \item if $x_i - x_j \in [3, 4]$, then the projection of the east blade of the west unbridge of $B_j$ intersects the projection of the east up bridge of $B_i$.
    \end{itemize}
    The border unbridges stay at height $0$ everywhere in $T^3$, so all their points are at height at most $\epsilon$ in $B_i$. The up bridges of $T^3$ are at height $h$ for at least two thirds of the south-north distance of $T^3$, so in $B_j$ they are at height at least $h - \epsilon$. From $B_i < B_j$ and $|y_i - y_j| \leq 2 e$ we obtain $|z_i - z_j| \geq h - 2\epsilon$.
  \end{proof}

  It follows from Claim~\ref{cl:xDistance} that if all pairs $i \neq j$ satisfied $|x_i - x_j| \in [0,4]$, then we would have $|z_i - z_j| \geq h - 2\epsilon$ for each of them. But it is a simple observation about the geometry of the real line that if $|z_i - z_j| \geq x = h - 2\epsilon$ for all $i \neq j \in \{1,2,3\}$, then some pair satisfies $|z_i - z_j| \geq 2x = 2h - 4\epsilon$. If $B_i < B_j$, then the highest point of $B_i$ is at height at most $z_i + h + \epsilon$, and the lowest point of $B_j$ is at height at least $z_j - b - \epsilon$.
  We now have
  \[ (z_j - b - \epsilon) - (z_i + h + \epsilon) 
    \geq h - b - 6\epsilon \geq \epsilon, \]
  so in particular the points $(0,0,r_i) \in B_i$ and $(0,0,r_j) \in B_j$ are at distance at least $\epsilon$ from each other.
  This contradicts the assumption $|r_i - r_j| < \epsilon$.

  We have shown that there exist $i, j$ such that $|x_i - x_j| \in (4,5]$. By left-right symmetry we may assume $x_i - x_j \in (4,5]$, i.e. $B_i$ is to the west of $B_j$. Then the points $(0,0,r_i)$ and $(0,0,r_j)$ lie on the east unbridge of $B_i$ and the west unbridge of $B_j$. Since the unbridges stay at height $0$ in $T^3$, we must have $|z_i - z_j| \leq 3 \epsilon$, and by vertical translation we may assume $z_i, z_j \in [0, 3\epsilon]$. Let $B_k$ be the third mat, i.e. $\{i,j,k\} = \{1,2,3\}$. Then we have $r_i, r_j, r_k \in [-2\epsilon, 5\epsilon]$.

  \begin{claim}
    \label{cl:Below}
    We have $B_k < B_i$ and $B_k < B_j$.
  \end{claim}

  \begin{proof}
    Suppose $B_k$ is above $B_i$, the other case being analogous. If $|x_i - x_k| \in [0,4]$, then by Claim~\ref{cl:xDistance} we have $z_k - z_i \geq h - 2\epsilon$, i.e. $z_k \geq z_i + h - 2\epsilon$. 
    The lowest point of $B_k$ is then at height at least
    \[ z_k - b - \epsilon \geq z_i + c - 3\epsilon > 5 \epsilon, \]
    which contradicts the observation $r_k \in [-2\epsilon, 5\epsilon]$.

    Suppose then $|x_i - x_k| \in (4,5]$. Since $x_i - x_j \in (4,5]$ and the projections of $B_j$ and $B_k$ intersect, we must also have $x_i - x_k \in (4,5]$, so that $|x_j - x_k| < 1$. The points $(0,0,r_j)$ and $(0,0,r_k)$ then lie on the left unbridges of $B_j$ and $B_k$, so that $|z_j - r_j|, |z_k - r_k| < \epsilon$, from which it follows that $|z_j - z_k| < 3 \epsilon$.
    On the other hand, Claim~\ref{cl:xDistance} implies $|z_j - z_k| \geq h - 2 \epsilon$, a contradiction with $\epsilon < h / 7$.
  \end{proof}

  We now have two mats $B_i$ and $B_j$ whose projections intersect only at their border unbridges, and a third mat $B_k$ below both of them. 
  We also have $x_i \in [4,5]$ and $x_j \in [0,1]$, and $z_i, z_j \in [0, 3 \epsilon]$ as already stated. Then $|x_i - x_k| \leq 4$ or $|x_j - x_k| \leq 4$, and in either case Claim~\ref{cl:xDistance} implies $z_k \leq 5 \epsilon - h$.
  Since the up bridges are the only parts of $T^3$ with positive height, we must have $x_k \in [1,2] \cup [3,4]$, that is, the point $(0,0,r_k)$ is on an up bridge of $B_k$. 

  By the left-right symmetry of the situation, we may assume $x_k \in [1, 2]$, so that $(0,0,r_k)$ is somewhere on the west up bridge of $B_k$. Then the horizontal projection of the east up bridge of $B_k$ intersects that of the down bridge of $B_j$. Since $d > \frac{2e}{3}$, we can find a horizontal coordinate $(x,y) \in \R^2$ that is on the projection of the top part of the east up bridge of $B_k$ and on the projection of the bottom part of the down bridge of $B_j$.

  Every point on the top part of the east up bridge of $B_k$ is at height at least $z_k + h - \epsilon \geq - 4 \epsilon$, while every point on the bottom part of the down bridge of $B_j$ is at height at most $-b + \epsilon < -4 \epsilon$. Since $B_k < B_j$, this contradicts Lemma~\ref{lem:Fundamental}. This finishes our case analysis, as we have shown that our assumptions were contradictory.
\end{proof}

We note that for getting the result $\underline \alpha(X) < \infty$, one could replace the mat from this section by essentially any mat with a sufficiently chaotic set of bridges. However, to get $\underline \alpha(X) = 2$ (the best we are able to obtain), one has to be quite careful; for example, the mats in Figure~\ref{fig:DoNotWork}, and obviously versions where up-slopes and down-slopes are exchanged, \emph{do not work} no matter how the lengths $a,b,c,d$ are chosen, in the sense that Lemma~\ref{lem:T3Lemma} is false, and thus one would at most obtain $\underline \alpha(X) \leq 3$ with our proof, for similar reasons as those outlined in the proof of Proposition~\ref{prop:SimplerExample} in Section~\ref{sec:SimplerExample}.

\begin{figure}
\[
\begin{smallmatrix}
\tile{ } & \tile{ } & \tile{ } & \tile{ } & \tile{ } \\
\rtile{ } & \lrtile{s} & \lrtile{n} & \lrtile{s} & \ltile{ } \\
\rtile{ } & \lrtile{ } & \lrtile{n} & \lrtile{ } & \ltile{ } \\
\rtile{ } & \lrtile{ } & \lrtile{ } & \lrtile{ } & \ltile{ } \\
\rtile{ } & \lrtile{ } & \lrtile{s} & \lrtile{ } & \ltile{ } \\
\rtile{ } & \lrtile{n} & \lrtile{s} & \lrtile{n} & \ltile{ } \\
\tile{ } & \tile{ } & \tile{ } & \tile{ } & \tile{ } \\
\end{smallmatrix}
\;\;\;\;\;
\begin{smallmatrix}
\tile{ } & \tile{ } & \tile{ } & \tile{ } & \tile{ } \\
\rtile{ } & \lrtile{s} & \lrtile{n} & \lrtile{s} & \ltile{ } \\
\rtile{ } & \lrtile{s} & \lrtile{n} & \lrtile{s} & \ltile{ } \\
\rtile{ } & \lrtile{ } & \lrtile{ } & \lrtile{ } & \ltile{ } \\
\rtile{ } & \lrtile{n} & \lrtile{s} & \lrtile{n} & \ltile{ } \\
\rtile{ } & \lrtile{n} & \lrtile{s} & \lrtile{n} & \ltile{ } \\
\tile{ } & \tile{ } & \tile{ } & \tile{ } & \tile{ } \\
\end{smallmatrix}
\;\;\;\;\;
\begin{smallmatrix}
\tile{ } & \tile{ } & \tile{ } & \tile{ } & \tile{ } \\
\rtile{ } & \lrtile{s} & \lrtile{ } & \lrtile{s} & \ltile{ } \\
\rtile{ } & \lrtile{s} & \lrtile{ } & \lrtile{s} & \ltile{ } \\
\rtile{ } & \lrtile{ } & \lrtile{ } & \lrtile{ } & \ltile{ } \\
\rtile{ } & \lrtile{n} & \lrtile{ } & \lrtile{n} & \ltile{ } \\
\rtile{ } & \lrtile{n} & \lrtile{ } & \lrtile{n} & \ltile{ } \\
\tile{ } & \tile{ } & \tile{ } & \tile{ } & \tile{ } \\
\end{smallmatrix}
\;\;\;\;\;
\begin{smallmatrix}
\tile{ } & \tile{ } & \tile{ } & \tile{ } & \tile{ } \\
\rtile{ } & \lrtile{s} & \lrtile{s} & \lrtile{s} & \ltile{ } \\
\rtile{ } & \lrtile{ } & \lrtile{s} & \lrtile{ } & \ltile{ } \\
\rtile{ } & \lrtile{ } & \lrtile{ } & \lrtile{ } & \ltile{ } \\
\rtile{ } & \lrtile{ } & \lrtile{n} & \lrtile{ } & \ltile{ } \\
\rtile{ } & \lrtile{n} & \lrtile{n} & \lrtile{n} & \ltile{ } \\
\tile{ } & \tile{ } & \tile{ } & \tile{ } & \tile{ } \\
\end{smallmatrix}
\]
\caption{Mat structures for which Lemma~\ref{lem:T3Lemma} does not hold.}
\label{fig:DoNotWork}
\end{figure}

\section{The two-dimensional substitution}
\label{sec:ThreeSparse}

We now consider a two-dimensional substitution which generates descriptions of mats like those studied in the previous section. Let $a,b,c,d \in \N$, let
\[ \Sigma = \{\ltile{ }, \tile{ }, \rtile{ }, \lrtile{ }, \ltile{n}, \tile{n}, \rtile{n}, \lrtile{n}, \ltile{s}, \tile{s}, \rtile{s}, \lrtile{s} \}, \]
and let $\tile{x}$ be a variable ranging over tiles $\{\tile{ }, \tile{n}, \tile{s}\}$. The following is the image of $\tile{x}$ under $\tau$:
\[\begin{smallmatrix}
\tile{x} & \tile{x} & \tile{x} & \tile{x} & \tile{x} & \mbox{ $\}$ $a$ times} \\
\rtile{ } & \lrtile{s} & \lrtile{n} & \lrtile{s} & \ltile{ } & \mbox{ $\}$ $b$ times} \\
\rtile{ } & \lrtile{s} & \lrtile{ } & \lrtile{s} & \ltile{ } & \mbox{ $\}$ $c$ times} \\
\rtile{ } & \lrtile{ } & \lrtile{ } & \lrtile{ } & \ltile{ } & \mbox{ $\}$ $d$ times} \\
\rtile{ } & \lrtile{n} & \lrtile{ } & \lrtile{n} & \ltile{ } & \mbox{ $\}$ $c$ times} \\
\rtile{ } & \lrtile{n} & \lrtile{s} & \lrtile{n} & \ltile{ } & \mbox{ $\}$ $b$ times} \\
\tile{x} & \tile{x} & \tile{x} & \tile{x} & \tile{x} & \mbox{ $\}$ $a$ times} \\
\end{smallmatrix}\]
For tiles with jagged west borders, i.e. $\ltile{x} \in \{\ltile{ }, \ltile{n}, \ltile{s}\}$, we define $\tau(\ltile{x})$ by adding a west border to each of the tiles in the westmost south-north column of $\tau(\tile{x})$, and similarly for east borders.

Here, as in the previous section, $a,b,c,d \in \N$ are constants that are chosen later, and again let $e = 2(a+b+c)+d$ and $h = b+c$. The idea is that for suitable choices, the iterated substitutions $\tau^k$ will generate mats that, once rescaled, are close to $T$ in the mat metric.
Since the macrotiles $\tau^k(s)$ for $s \in \Sigma$ consist of copies of $\tau^{k-1}(s')$ for $s' \in \Sigma$, they will be perturbed versions of $T$, with the perturbation parameter $\epsilon$ depending on $a, b, c, d$.
We will show that as long as the parameters are chosen suitably, the results of Section~\ref{sec:Mats} apply to these surfaces.
Concretely, we will see that the choice
\begin{align}
  \label{eq:GoodParams}
  a & {} = 16, & b & {}= 1316, & c & {}= 1504, & d & {} = 11344
\end{align}
works, with $\epsilon = 187$.

Any rectangle $R \in \Sigma^{m \times n}$ colored with the tiles in such a way that the slopes are consistent (i.e. the total slope of any cycle that does not cross jagged edges is $0$) can be interpreted as a three-dimensional surface $\Surf(R) \subset \R^3$ in at least one way, and up to translation, there is a unique such way if $R$ is connected by non-jagged tile edges.
In particular, this is the case for $S = \Surf(\tau^k(s))$ for all $k \in \N$ and $s \in \Sigma$.
We call $S$ a \emph{geometric macrotile} of level $k$.
We fix the default position of $S = \Surf(\tau^k(s))$ so that its southwest corner is at the origin if $k = 0$, and at $(0, 0, -a (2 a)^{k-1})$ if $k \geq 1$.
In the latter case, this implies that the z-coordinate of each flat tile on the east and west border of $S$ is $0$.
We say that the \emph{baseline} of a translate $(x,y,z) + S$ is $z$.

\begin{lemma}
  \label{lem:SufficientParams}
  Let $a,b,c,d \in \N$ satisfy $a > 0$ and
  \begin{align}
    \frac{b + c}{2a-1} + 1 + 6a & {} < \min(b/7, c/8), \label{eq:bcbound} \\
    d & {} \geq 4(a+b+c). \label{eq:dbound}
  \end{align}
  Denote $e = 2(a+b+c) + d$.
  Then for all $n \geq 1$ and any combination of three tiles $A \in \Sigma^{1 \times 3}$ the following hold:
  \begin{enumerate}
  \item \label{item:1}The surface $M = \Surf(\tau^n(A))$, rescaled by $(5^{-n+1}, e^{-n+1}, (2a)^{-n+1})$ axiswise and translated suitably, is a mat with projection $[0,5] \times [-e,2e]$.
  \item \label{item:2} $\epsilon = \frac{b + c}{2a-1} + 1 + 6a$ satisfies the assumptions of Lemma~\ref{lem:T3Lemma} for the surface $T^3$ corresponding to these choices of $a,b,c,d$.
  \item \label{item:3} $M$ is $\epsilon$-close to $T^3$ in the mat metric.
  \end{enumerate}
\end{lemma}

Note that it is easy to satisfy the assumptions of this lemma, for example by choosing  $2a-1 > 32$, $b = c > 16(6a + 1)$ and then $d \geq 4(a+b+c)$. The unique choice that satisfies them and minimizes the south-north length of the tile is~\eqref{eq:GoodParams}, with length $e = 17016$, and then Lemma~\ref{lem:T3Lemma} applies with $\epsilon = 187$. We provide these numbers only for concreteness; much smaller ones could presumably be obtained by optimizing Lemma~\ref{lem:T3Lemma}.

\begin{proof}
Note that we need $a > 0$ since otherwise the scaling is not well-defined. Item~\ref{item:1} follows because the images of tiles in the substitution $\tau^n$ have constant size $(5^n, e^n)$.
Since~\eqref{eq:dbound} is equivalent to $d \geq \frac{2e}{3}$ and we explicitly assume~\eqref{eq:bcbound} for $\epsilon$, the assumptions of Lemma~\ref{lem:T3Lemma} are satisfied, showing item~\ref{item:2}.

We now show item~\ref{item:3}. Note first that the vertical scaling factor $2a$ is chosen so that in the scaled surface corresponding to the level-$n$ macrotile $\tau^n\left(\begin{smallmatrix} \smalltile{ } \\
\smalltile{ } \\
\smalltile{ }
\end{smallmatrix}\right)$, the south and north borders of all the constituent $(n-1)$-tiles are precisely at the same height as in the $T^3$, which is easily shown by induction. Thus, we only need to study how far we get from $T^3$ due to the effects of other choices of $A$ and the internal height fluctuations in $(n-1)$-tiles.


Define $r_n = (2a)^n$ for $n \geq 0$, which is the absolute vertical difference between the north and south borders of $\Surf(\tau^n(\tile{n}))$ and $\Surf(\tau^n(\tile{s}))$ and their jagged variants. Let $h_n$ be the maximum absolute vertical difference between any point of a geometric level-$n$ macrotile and its baseline, so that $h_0 = 1$, $h_1 = b + c$ and
\[ h_n = (b + c) r_{n-1} + h_{n-1} = (b + c) \sum_{i = 0}^{n-1} r_i = (b + c) \frac{(2a)^n - 1}{2a-1} \]
for $n \geq 2$.
Then for any choice of $A$, the rescaled surface $M$ is at distance $\epsilon' = h_{n - 1}/(2a)^{n-1} + 1 + 6a$ from $T^3$.
Here $6a$ comes from the ascent or descent by $2a$ in each scaled $n$-tile, depending on the tile type, $h_{n - 1}/(2a)^{n-1}$ comes from internal fluctuations in the constituent $(n-1)$-tiles with respect to their baseline, and $1$ is an upper bound on the difference between the baseline of a rescaled constituent $(n-1)$-tile and the height of the corresponding position in $T^3$.
To conclude, observe that
\[ h_{n - 1}/(2a)^{n-1} = (b + c) \frac{(2a)^{n-1}-1}{(2a-1)(2a)^{n-1}} \leq \frac{b + c}{2a-1} \]
so $\epsilon' \leq \frac{b + c}{2a-1} + 1 + 6a = \epsilon$.
\end{proof}

\section{Zero-gluing closures}

Let $G$ be a group with some fixed left-invariant metric. The \emph{$r$-connected components} of a configuration $x \in \Sigma^G$ are the following: Consider nonzero cells in $x$ to be vertices in a graph and put edges between vertices at word-distance $r$ to obtain a graph. Then the connected components of $x$ are the subsets of $G$ that are connected components in this graph. The \emph{$r$-connected component configurations} of $x$ are the configurations where all but one $r$-connected components are turned to $0$. We usually omit $r$ if it is obvious from context, and we often identify connected components and the corresponding configurations, which should not cause confusion. We denote the set of $r$-connected component configurations of $x \in \Sigma^G$ by $C_r(x) \subset \Sigma^G$

In our example, the group is $\Z^3$ and the metric is the $\ell_\infty$ norm,\footnote{If we use the word norm for the standard generators, we need to use either $r = 2$, or talk about essential sparseness instead of sparseness, as $r = 1$ with the word metric would add artificial width in the construction, as it does not allow a surface to move vertically without locally stacking two nonzero cells.} and $r = 1$. A subshift $X \subset \Sigma^G$ is \emph{$r$-zero-gluing} if 
\[ \forall x \in \Sigma^G: (x \in X \iff \forall z \in C_r(x): z \in X), \]
that is, $x \in \Sigma^G$ is in $X$ if and only if its $r$-connected component configurations are in $X$.

We say two vertices or sets of vertices \emph{see} each other if the minimal distance between them is at most $r$. Then equivalently, a zero-gluing subshift is one where we can freely turn a connected component to $0$ without affecting other components, and where we can conversely glue two configurations together whenever their $r$-supports do not see each other.

Let $X \subset \Sigma^G$ be a subshift, where $0 \in \Sigma$. Then the corresponding full shift $\Sigma^G$ is zero-gluing with radius $r$ for all $r$. For a given $r$, the intersection of all $r$-zero-gluing subshifts containing $X$ is the \emph{$r$-zero-gluing closure of $X$}, denoted $\ZGC(X)$.

\begin{lemma}
  \label{lem:GluingIsGluing}
  The $r$-zero-gluing closure of a subshift is $r$-zero-gluing.
\end{lemma}

\begin{proof}
  Let $Y$ be the $r$-zero-gluing closure of a subshift $X$, and suppose $y \in Y$ and that $c$ is an $r$-connected component configuration of $y$. If $Z$ is $r$-zero-gluing and contains $X$, then $y \in Y$ implies $y \in Z$. Since $Z$ is $r$-zero-gluing, we have $c \in Z$. But then $c$ is in the intersection of all such $Z$, so $c \in Y$. Conversely, suppose that $y \in \Sigma^G$ and every $r$-component $c$ of $y$ is in $Y$. If $Z$ is $r$-zero-gluing and contains $X$, then every $r$-component of $y$ is in $Z$, so by definition also $y$ is in $Z$. But then $y$ is in the intersection of all such $Z$, thus in $Y$.
\end{proof}

Concretely, the $r$-zero-gluing closure of a subshift $X$ is the subshift where a configuration is valid if and only if its $r$-connected component configurations are $r$-connected component configurations of confiugrations in $X$.

\begin{lemma}
  \label{lem:00SoficGluing}
  Let $X \subset A^G$ be $0$-to-$0$ sofic with $0$-to-$0$ SFT cover $\phi : Y \to X$ such that $Y$ is defined by forbidden patterns of diameter at most $n$ and $\phi$ has a neighborhood of diameter $d$.
  Then $X$ is $(n+d-1)$-zero-gluing.
\end{lemma}

\begin{proof}
  Let $x \in A^G$ be arbitrary.
  If $x \in X$, then there exists a preimage $y \in \phi^{-1}(x)$.
  Since $\phi$ is $0$-to-$0$, we have $\supp(y) \subset \supp(x) \cdot N$ where $N \subset G$ is the neighborhood of $\phi$.
  Let $H$ be an $(n+d-1)$-connected component of $x$.
  Then $\supp(y) \cap H$ and $K := \supp(y) \setminus H$ have distance at least $n-1$.
  By our choice of $n$ we can replace $y|_K$ by $0$-symbols and obtain a new configuration $z \in Y$.
  Then $\phi(z) \in X$ is the configuration obtained from $x$ by replacing $G \setminus H$ by $0$-symbols.

  Conversely, suppose that for every $(n+d-1)$-connected component $H$ of $x$, changing $x|_{G \setminus H}$ into $0$-symbols results in some $z^H \in X$.
  Then $z^H$ has a preimage $y^H \in Y$ with $\supp(y^H) \subset H \cdot N$, and since these sets are $n-1$-separated for distinct $H$, we can glue the $y^H$ into a single configuration $y \in Y$ with $y|_{H N} = y^H|_{H N}$ for each $H$.
  Then $\phi(y)|_H = \phi(y^H)|_H = z^H|_H = x|_H$ and for each $g \in G$ outside these components, $\phi(y)_g = 0 = x_g$.
  Thus $x = \phi(y) \in X$.
\end{proof}

\section{The sparse subshift $Z$}
\label{sec:DefOfZ}

In this section we define our example $Z$ of a zero-gluing subshift and prove its projective subdynamics is 2-sparse but not 1-sparse. In Section~\ref{sec:SFT} we prove that it is a $0$-to-$0$ sofic image of a deterministic SFT, and in Section~\ref{sec:WangCubes} we implement it using Wang cubes, obtaining the example claimed in the abstract.

We represent the surface of each tile of $\Sigma$ as a three-dimensional cube so that a flat surface is represented on the bottom of a cube, and north-south ascending and descending surface pieces connect two opposite edges of the cube.
We add a fourth blank cube that contains no surface to obtain an alphabet $\Delta$ of size 4 and a non-surjective map $\psi : \Sigma \to \Delta$.
See Figure~\ref{fig:SlopeCubes}.

\begin{figure}[ht]
  \begin{center}
    \begin{tikzpicture}

      \foreach \x/\y/\i/\o/\z in {
        9/0/0/0/0,8/0/1/0/0,7/1/1/0/0,6/2/0/0/0,5/2/0/0/0,
        4/2/0/0/0,3/1/0/1/0,2/0/0/1/0,1/0/0/0/0,0/0/0/0/0,
        9/0/0/0/1,8/0/0/0/1,7/-1/0/1/1,6/-1/0/0/1,5/-1/0/0/1,
        4/-1/0/0/1,3/-1/1/0/1,2/0/0/0/1,1/0/0/0/1,0/0/0/0/1}{
        \draw (\x+\z/3,\y-\z/5) rectangle ++(1,1);
        \draw [fill=black!25] (\x+\z/3,\y-\z/5+\i) -- (\x+\z/3+1,\y-\z/5+\o) -- (\x+\z/3+1+1/3,\y-\z/5+\o-1/5) -- (\x+\z/3+1/3,\y-\z/5+\i-1/5) -- cycle;
        \draw (\x+\z/3,\y-\z/5) -- ++(1/3,-1/5) -- ++(1,0) -- ++(-1/3,1/5);
        \draw (\x+\z/3,\y-\z/5+1) -- ++(1/3,-1/5) -- ++(1,0) -- ++(-1/3,1/5);
        \draw (\x+\z/3+1/3,\y-\z/5-1/5) -- ++(0,1);
        \draw (\x+\z/3+1+1/3,\y-\z/5-1/5) -- ++(0,1);
      }
      
    \end{tikzpicture}
  \end{center}
  \caption{Representation of sloped tiles by cubes. The south-north axis is horizontal, east is toward the viewer and the z-axis is vertical. Blank cubes are not shown.}
  \label{fig:SlopeCubes}
\end{figure}
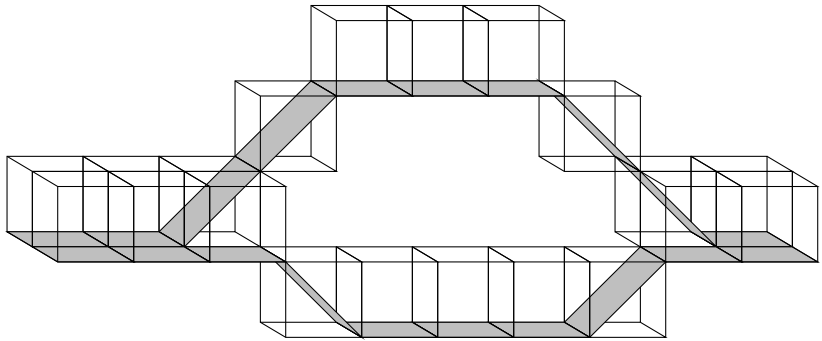

Pick the substitution $\tau$ so that the assumptions of Lemma~\ref{lem:SufficientParams} hold.
Now interpret the surface $\Surf(\tau^n(\tile{ }))$ as a finite-support configuration $x^n \in \Delta^{\Z^3}$, translated so that the $\psi$-image of the central tile is at the origin $(0,0,0)$.
Observe that the sequence $(x^n)$ converges as $n \rightarrow \infty$ and call this limit configuration $x \in \Delta^{\Z^3}$.

\begin{definition}
  The subshift $Z \subset \Delta^{\Z^3}$ is the $1$-zero-gluing closure of $\OC{x}$.
\end{definition}

\begin{theorem}
  \label{thm:ZIsSparse}
The subshift $Z$ is 1-zero-gluing and has $2$-sparse $\Z$-projective subdynamics.
\end{theorem}

\begin{proof}
  The first claim follows from Lemma~\ref{lem:GluingIsGluing}.
  For the second one, suppose there exist $y \in Z$ and distinct coordinates $\vec v_1, \vec v_2, \vec v_3 \in \supp(y)$ with identical x- and y-coordinates.
  For each $n \geq 0$, each $\vec v_i$ is part of a unique level-$n$ geometric macrotile $K^i_n$, which has other level-$n$ geometric macrotiles to its north and south.
  Denote by $M^i_n$ the union of these three macrotiles.
  Suppose the z-coordinates of the $\vec v_i$ differ pairwise by at most $m$, and pick $n$ so that
  $\epsilon = \min(b/7, c/8) > m/(2a)^{n-1}$. 
  By Lemma~\ref{lem:SufficientParams}, the macrotile triples $M^i_n$ are $\epsilon$-close to $T^3$ after rescaling, so by Lemma~\ref{lem:T3Lemma} some two of them intersect, say $M^i_n$ and $M^j_n$ for $i \neq j$. Since the geometric macrotiles are connected, this implies that $\vec v_i$ and $\vec v_j$ lie in the same 1-connected component of $\supp(y)$.
  But $\OC{x}$ does not have such configurations and the 1-zero-gluing closure operation cannot introduce them, so we have reached a contradiction.
\end{proof}

\section{SFT Implementation}
\label{sec:SFT}

We show that $Z$ from the previous section is $0$-to-$0$ sofic.
We choose a number $N \geq 1$, whose exact value will be determined later, and implement the 3-dimensional version of the substitution $\tau^N$ by a self-simulating SFT $X$ that is deterministic to the north and projects onto $Z$ by a $0$-to-$0$ symbol map.
We then distort it to obtain a sparse SFT that is deterministic along the z-axis.
While these SFTs are neither binary nor defined by Wang cubes, we show in Section~\ref{sec:Variants} that the construction can be modified to ensure these properties.

The construction follows the ideas of programmatic self-simulation presented in~\cite{Ga01,DuRoSh12}, but we have the additional difficulty that we can only communicate through some parts of the configuration (as some adjacent tiles of the $\Z^2$-substitution system correspond to cells with large vertical offsets in the $\Z^3$-subshift). This is related to the phenomenon studied in \cite{BaSa16}.

Since the construction is relatively complex, we list its key properties as lemmas as soon as they become apparent from the construction.
We note that~\cite{Mo89} presents a conceptually simpler way of implementing substitutive subshifts as sofic shifts, but requires more technical modifications to produce a deterministic SFT cover.

\subsection{The alphabet and local structure of $X$}

The alphabet of the SFT is denoted by $\Gamma$, which contains the special symbol $0$.
Nonzero symbols of $\Gamma$ in a configuration of $X$ are called \emph{cells}.
Each cell $g$ contains the following pieces of information, called \emph{fields}:
\begin{itemize}
\item Three elements of $\Sigma$, called the \emph{past, current and future preimage fields} of~$g$.
\item A coordinate $\vec{v} \in [0, 5^N-1] \times [0, e^N-1]$, called the \emph{address field}.
\item An element of $\{0,1,\#\}$, called the \emph{program field}.
\item An element of $\{0,1,\#\}$, called the \emph{simulation field}.
\item A tape symbol $t \in \Sigma_M$ and possibly a state $q \in Q_M$ of a Turing machine $M$, called the \emph{computation field}.
\item Two elements of $\{0,1,\#\}$, called the \emph{west and east mailbox fields}.
\end{itemize}
If the current preimage field of $g$ holds a symbol $s$ and its address field holds $\vec{v}$, then we say $g$ \emph{simulates} the symbol $r = \tau(s)_{\vec{v}}$ and denote $r = S(g)$.
The SFT cover $\phi : X \to Z$ will be defined cellwise by $\phi(g) = \psi(S(g))$ for each $g \in \Gamma \setminus \{0\}$ and $\phi(0)$ being the empty cube, where $\psi : \Sigma \to \Delta$ is the map defined in Section~\ref{sec:DefOfZ}.
We fix a scheme to encode each element of $\Gamma \setminus \{0\}$ as a binary string of some fixed length $\ell$, which contains each field as a fixed-length substring at a fixed position.
The behavior of the Turing machine $M$ will be determined later.

We define the adjacency rules of $X$ in a way that explicitly determines the symbol of a coordinate $(i,j,k) \in \Z^3$ from $(i+h, j-1, k+h')$ for $h, h' \in \{-1,0,1\}$ in a valid configuration $x \in X$, and give some additional rules that forbid certain configurations.
All forbidden patterns will have diameter at most 1 in the $\ell_\infty$ norm.
For this, the \emph{north neighbor coordinates} of a coordinate $(i,j,k) \in \Z^3$ are $(i,j+1,k+h)$ for $h \in \{-1,0,1\}$.
South neighbor coordinates are defined similarly.
The \emph{east neighbor coordinate} of $(i,j,k)$ is $(i+1,j,k)$ and its \emph{west neighbor coordinate} is $(i-1,j,k)$.
Note the asymmetry between the x- and y-axes, which is due to the fact that the alphabet of $\tau$ has slopes and jagged edges only along the y-axis.

Suppose we have a cell $g \in \Gamma \setminus \{0\}$ at position $\vec v = (i,j,k)$ of $x$.
Then there must be exactly one cell $g'$ among its north neighbor coordinates, and its z-coordinate is determined by the slope of the simulated symbols $S(g)$ and $S(g')$: $k-1$ if $S(g')$ has downward slope, $k+1$ if $S(g)$ has upward slope, and $k$ otherwise.
There must also be exactly one cell among the south neighbor coordinates of $g$.
The west neighbor coordinate of $g$ should contain a cell if and only if $S(g)$ has no jagged west edge, and similarly for the east neighbor coordinate.
We call these cells, if they exist, the \emph{north, south, west and east neighbors} of $g$.
There must be no other cells in $\vec v + [-1,1]^3$, that is, within $\ell_\infty$-distance 1 from $\vec v$.



In the remaining part of the construction, we will think of the y-axis as ``time'', and the north neighbor of a cell represents that cell one step later in time.
Thus, for example, ``cell $g$ copies field $F$ from its east neighbor'' means that field $F$ of the north neighbor of $g$ must equal field $F$ of the east neighbor of $g$.
We introduce only two kinds of rules: \emph{deterministic rules} enforce new values of some fields of $g$ based on the current values of $g$ and its west and east neighbors, and \emph{constraints} forbid some combinations of field values of $g$ and possibly its east neighbor.

\subsection{Level-1 structure}

We first handle the address and current and past preimage fields.
At every time step, each cell $g$ changes its address from $(p,q)$ to $(p,q')$ where $q' \equiv q+1 \bmod e^N$.
If $q < e^N-1$, then the past and current preimage fields of $g$ do not change.
If $q = e^N-1$, then we require that they change to the current and future preimage fields of $g$, respectively.
If $g$ has an east neighbor, its address is constrained to be $(p',q)$ where $p' \equiv p+1 \bmod 5^N$.
If $p < 5^N-1$, then this neighbor is also constrained to have the same current preimage field as $g$.
If $p = 5^N-1$, $q = e^N-1$ and $g$ has an east neighbor $g'$, then we require that the $2 \times 3$ pattern formed by preimage fields of $g$ and $g'$ in the order
\[
  \begin{array}{rr}
    \text{future}(g) & \text{future}(g') \\
    \text{current}(g) & \text{current}(g') \\
    \text{past}(g) & \text{past}(g')
  \end{array}
\]
occurs in the subshift $X_\tau \subset \Sigma^{\Z^2}$.
If $g'$ does not exist, then we require that the preimage fields of $g$ form a $1 \times 3$ pattern with the same property.
We call these the \emph{local preimage constraints}.

We inductively define an \emph{$X$-macrocell of level $n$} for $n \geq 0$, their \emph{neighbor relation}, \emph{base coordinates} and \emph{contents} as follows.
An $X$-macrocell of level $0$ is simply a cell, its base coordinate is its coordinate and its content is the value of $\Gamma$ it holds.
For $n \geq 1$, an $X$-macrocell of level $n$ is a collection of level-$(n-1)$ macrocells $K_{(p,q)}$ for $p \in [0, 5^N-1]$, $q \in [0, e^N-1]$ such that each $K_{(p,q)}$ has base coordinate $(i+p 5^{n N},j+q e^{n N},k+h_{p,q})$ for some $(i,j,k) \in \Z^3$ and $h_{p,q} \in \Z$, and its content has address field $(p,q)$ and common current preimage field $s \in \Sigma$.
We also require that if the tiles of $\tau^N(s)$ at adjacent coordinates $\vec v$ and $\vec v'$ do not share a jagged edge, then $K_{\vec v}$ and $K_{\vec v'}$ are neighbors.
If $K$ and $K'$ are two $X$-macrocells of level $n$ and some level-$(n-1)$ macrocell of $K'$ is the east neighbor of one of $K$, then we say $K'$ is an east neighbor of $K$.
Neighbor relations in the other directions are defined analogously.

Recall that we encode cells of $\Gamma$ with binary strings of length $\ell$.
Given an $X$-macrocell $K$ of level $n$, consider the contents of its constituent level-$(n-1)$ macrocells $K_{(0,0)}, K_{(0,1)}, \ldots, K_{(0,\ell-1)}$, which are elements of $\Gamma$.
The binary string formed by the simulation fields of these symbols, if it encodes a symbol of $\Gamma$, is the content of $K$, denoted $c(K)$.
Otherwise the content is undefined (this situation will become impossible later).

\begin{lemma}
  \label{lem:Level1Macrocells}
  Every cell of a configuration $x \in X$ is part of a unique $X$-macrocell of level 1.
\end{lemma}

\begin{proof}
  Let $g$ be a cell in $x$ with current preimage $s$.
  Since every cell has exactly one north and south neighbor, $g$ is part of a unique sequence of neighboring cells $g_0, \ldots, g_{e^N-1}$ with addresses $(p,0), (p,1), \ldots (p,e^N-1)$ for some $p \in [0, 5^N-1]$ and current preimage $s$, running from south to north.
  The z-coordinates of these cells are given by the slopes of the corresponding tiles of $\tau^N(s)$.
  The southmost row of $\tau(s)^N$ is connected, and because the absence of jagged edges enforces the existence of neighbors in cells of $x$, the cell $g_0$ is part of a unique west-east sequence of neighboring cells $g_{(0,0)}, g_{(1,0)}, \ldots, g_{(5^N-1,0)}$ with respective addresses and identical z-coordinates.
  Like $g$, each of these cells is part of a unique south-north sequence of cells, resulting in a connected set $K$ of cells $g_{(p,q)}$ for $p \in [0,5^N-1]$ and $[0,e^N-1]$ with current preimage $s$.
  
  Since every directed cycle in $\tau^N(s)$ has slope $0$, the total slope between two tiles does not depend on the choice of the path.
  Hence the relative z-coordinates of cells in $K$ correspond to the relative slopes between tiles in $\tau^N(s)$ and the non-jagged adjacency relation of $\tau^N$ corresponds to the neighbor relation between the respective cells of $K$.
  Thus $K$ is an $X$-macrocell of level 1.
\end{proof}

\begin{lemma}
  \label{lem:NeighborMacrocells}
  Let $K$ and $K'$ be $X$-macrocells of level 1.
  If $K'$ is the east neighbor of $K$, then for each $q \in [0, e^N-1]$, the cell $K'_{(0,q)}$ is the east neighbor of $K_{(5^N-1,q)}$, and there are no other neighbor relations between cells of $K$ and $K'$.
  In particular, $K$ has at most one east neighbor.
  Analogous claims hold for north, west and south neighbors.
\end{lemma}

\begin{proof}
  The only cells of $K$ that may have an east neighbor not in $K$ are those on its east border with addresses $(5^N-1,q)$, as the others either have east neighbors in $K$ or correspond to tiles with jagged east edges.
  The constraints on the addresses of neighboring cells guarantee that their east neighbors have addresses $(0,q)$.
  Furthermore, if one of these cells has an east neighbor, then the common current preimage of $K$ has no jagged east edge, so all of them must have an east neighbor.
  These neighbors must come from the same $X$-macrocell, since for each $q \in [1,e^N-1]$ the east neighbor of $K_{(5^N-1,q)}$ must be the unique north neighbor of $K'_{(0,q-1)}$, which is $K'_{(0,q)}$.
  The other directions are handled analogously.
\end{proof}

\subsection{Computation and simulation}

We now impose additional rules and constraints on $X$ to enforce the higher-level substitutive structure.
The idea is to use the remaining fields to simulate a computation that enforces the structural rules of $X$ on macrocells and lifts itself to the next-level macrocells, resulting in an infinite hierarchy of simulations.
For this, recall from Section~\ref{sec:ThreeSparse} that $e = 2a + 2b + 2c + d$, and that in each image $\tau(s)$, the region $[0,4] \times [0, a-1]$ is free from jagged edges.
The \emph{computation set}, denoted $J$, is the subset of $[0, e^N-1]$ of those numbers whose base-$e$ representation consists of digits in $[0, a-1]$.
We define the \emph{computation grid} as $[0, 5^N-1] \times J$.
Suppose we have a cell $g$ with address $(p,q)$.
If $q+1 \in J$, then each cell on row $q+1$ of the $X$-macrocell $C$ containing $g$ (except the bordermost ones) has a west and east neighbor, so the cells with addresses $(p+h, q+1)$ for $h \in \{-1,0,1\}$ can receive information from $g$ by nearest-neighbor rules.
This means we have a subgrid of size $[0, 5^N-1] \times [0, a^N-1]$ inside each $X$-macrocell in which we can transfer and process information.
We denote by $\pi_J : [0, a^N-1] \to J$ the increasing bijection.

We describe the mailbox fields.
At time step $0 = \pi_J(0)$, each cell copies its simulation field into its west and east mailbox fields.
At each time step $q \in \pi_J([1, 5^N])$, each cell copies its west mailbox field from its west neighbor, and its east mailbox field from its east neighbor.
If these neighbors do not exist, the value of the copied field is $\#$.
At other time steps, each cell keeps its mailbox fields intact, and all other fields are kept intact at all time steps $q \leq 5^N$.
The effect of this definition is that the cells with addresses $\{p\} \times \pi_J(5^N+1)$ contain in their mailbox fields the simulation fields of the cells with the same addresses in the west and east neighbors of their $X$-macrocell of level 1, or $\#$ if those macrocells do not exist.

We describe the computation field.
At time step $\pi_J(5^N+1)$, the computation fields are initialized so that the cell with address $(0, \pi_J(5^N+1))$ contains the head in the initial state, other cells at that time do not contain the head, and the tape symbol of each of these cells is blank.
At each time step $\pi_J(q)$ for $5^N+1 < q < a^N-1$, each cell simulates one step of $M$ in the standard way.
At every time step not mentioned, the computation field retains its value.

We now describe the machine $M$, whose purpose is to simulate the local rules of $X$ inside each $X$-macrocell $K$.
It is allowed to read the program and mailbox fields of each cell of $C$ (but not modify them), and it may freely modify the contents of the simulation, tape, and future preimage fields of the cells of $K$.
The program fields of cells never change, and the simulation and future preimage fields never change unless modified by $M$.
The machine behaves like this:
\begin{enumerate}
\item
  Copy the longest prefixes over $\{0, 1\}$ of the program fields, west mailbox fields, simulation fields, and east mailbox fields of each cell into the tape in this order, so that they form four bitstrings $P, b_w, b_s, b_e \in \{0, 1\}^*$.
\item
  Simulating an efficient universal Turing machine $M_U$, execute the program $P$ on $(b_w, b_s, b_e)$, and write the resulting bitstring $b_r \in \{0, 1\}^*$ to the tape.
\item
  Verify that the length of $b_r$ is $\ell$ and it encodes a nonzero symbol $g \in \Gamma$. If not, then halt.
\item
  Write $b_r$ into the simulation fields of the $\ell$ leftmost cells and $\#$ to the simulation fields of the remaining cells.
\item
  Write the simulated symbol $S(g)$ in the future preimage fields of all cells.
\item
  If $g$ has address $(p,q)$, verify that the program bit of $g$ equals the $p$'th bit of the program $P$, or equals $\#$ if the program is shorter than this. If not, then halt.
\end{enumerate}
Each phase except the second one takes $O(5^{2 N})$ computation steps since the bitstring encoding contains each field in a consistent position.
If the machine $M$ has enough time and space to perform all phases, then in each $X$-macrocell it simulates a partial function encoded in the program fields of the cells.
We add to $X$ the constraint that the machines cannot halt or run out of time or space, so that the computations always succeed in all valid configurations.
We formalize these definitions into a lemma.

\begin{lemma}
  \label{lem:ProgramEffect}
  Let $K$ be an $X$-macrocell of level 1, and let $P, b_s \in \{0,1\}^*$ be the longest prefixes over $\{0,1\}$ of the program and simulation fields of $K_{(0,0)}, \ldots, K_{(5^N-1,0)}$.
  Define $b_w, b_e \in \{0,1\}^*$ similarly for the west and east neighbors of $K$ if they exist, or as empty strings if they do not.
  Then the content of the north neighbor $K'$ of $K$ is $c(K') = M_U(P, b_w, b_s, b_e) \in \{0,1\}^\ell$.
  The simulated tile $S(c(K')) \in \Sigma$ is the common current preimage field of the cells of $K'$ and the future preimage field of cells on the northmost row of $K$.
  The past preimage field of each cell of $K'$ equals the common current preimage field of cells of $K$.
\end{lemma}

\begin{proof}
  Lemma~\ref{lem:NeighborMacrocells} ensures that the north neighbor $K'$ exists, and that the possible west and east neighbors of $K$ exist if and only if $c(K)$ simulates a tile without the respective jagged edges and are correctly aligned with $K$.
  The claim then follows from the construction of $X$.
\end{proof}

\subsection{Bootstrapping the simulation structure}

It remains to ``bootstrap'' the simulation by fixing the program $P$ in each macrocell.
Let $P_X$ be a program that implements all the rules and constraints we have defined for $X$ so far, in the sense that $M_U(P_X)$ is a partial function from $(\Gamma \cup \{\#^\ell\}) \times \Gamma \times (\Gamma \cup \{\#^\ell\})$ to $\Gamma$ that, given the states of a cell and possibly its west and east neighbors, assigns a north neighbor to it or diverges if one cannot exist.
We add one more constraint to $X$ that is not encoded into $P_X$: the program field of each cell with address $(p, 0)$ must be equal to the $p$'th bit of $P_X$ if $p < n$, and to $\#$ if $p \geq n$.
This constraint bootstraps the simulation, ensuring that all $X$-macrocells of level 1 contain $P_X$ in the program fields of their cells and thus simulate the rules of $X$ and form $X$-macrocells of level 2.
The last phase of $M$ ensures that they too use the program $P_X$ and form level-3 macrocells, and so on.

\begin{lemma}
  \label{lem:EveryLevelMacrocells}
  For each $n \geq 0$, every cell in a configuration of $X$ belongs to a unique $X$-macrocell $K$ of level $n$.
\end{lemma}

\begin{proof}
  The case $n = 0$ is trivial and Lemma~\ref{lem:Level1Macrocells} implies case $n = 1$.
  Suppose then that $n \geq 2$ and let $g$ be a cell in a configuration $x \in X$.
  Then $g$ belongs to a unique $X$-macrocell $K'$ of level $n-1$.
  Lemma~\ref{lem:ProgramEffect} applied inductively to $n-1$ and $P = P_X$ implies that $X$-macrocells of level $n-1$ behave exactly as cells in terms of their possible contents and neighbor relations.
  Then the proof of Lemma~\ref{lem:Level1Macrocells} can be applied to level-$(n-1)$ macrocells to show that $K'$, and thus $g$, belongs to a unique $X$-macrocell of level $n$.
\end{proof}

\begin{lemma}
  \label{lem:MacrocellsCover}
  Let $x \in X$ and let $C \subset \Z^3$ be finite such that for each $\vec v \in C$, $x_{\vec v}$ is a cell, and these cells are connected in $x$ by the neighbor relation.
  Then the two-dimensional pattern $F = \{(x,y) \mapsto S(x_{\vec v}) \;|\; (x,y,z) = \vec v \in C \}$ over $\Sigma$ occurs in $X_\tau$.
\end{lemma}

\begin{proof}
  Let $D \subset \Z^3$ be a finite set of cells in $x$ that contains $C$ and is connected by the neighbor relation.
  By Lemma~\ref{lem:EveryLevelMacrocells}, for each $n \geq 0$ each cell in $D$ belongs to a unique $X$-macrocell of level $n$ in $x$.
  Let $\mathcal{K}_n$ be the set of these macrocells, and let $B_n$ be the set of their base coordinates.
  Since $D$ is connected by the neighbor relation of cells, $\mathcal{K}_n$ is connected by the neighbor relation of level-$n$ macrocells.
  
  Let $\pi_1, \pi_2 : \Z^3 \to \Z$ be the projection maps to x- and y-coordinates.
  Then $\min \pi_1(B_n) > \min \pi_1(D) - 5^{n N}$ and $\max \pi_1(B_n) \leq \max \pi_1(D)$, and since $\pi_1(B_n)$ lies in a single congruence class modulo $5^{n N}$, for large enough $n$ the set $\pi_1(B_n)$ has at most two elements.
  The same holds for $\pi_2(B_n)$.
  If $|\pi_1(B_n)| = 1$, then $C$ is contained in either a single macrocell or a pair of south-north neighbors, and Lemma~\ref{lem:ProgramEffect} and the local preimage constraints of $X$ ensure that their simulated tiles form a $1 \times 2$ pattern occurs in $X_\tau$.
  If $|\pi_1(B_n)| = 2$, then $\mathcal{K}_n$ consists of at most 4 macrocells, some two of which are west-east neighbors, say $K_w$ and $K_e$.
  Because of the local preimage constraints, the level-$(n-1)$ macrocell $K_{(5^N-1,0)}$ of the north neighbor $K$ of $K_w$ guarantees that the simulated tiles of $\mathcal{K}_n$ form a pattern of $X_\tau$, as it is part of a valid $2 \times 3$ pattern.
  In either case, $F$ occurs in the $\tau^n$-image of the pattern and hence in $X_\tau$.
\end{proof}

It remains to show that for some values of $e$ and $N$ and some choice of the program $P_X$, the machine $M$ has enough time and space to simulate the universal machine $M_U$ on all valid inputs, so that $X$ is nontrivial.
We first observe that the length $\ell$ of the encoding is logarithmic in $e^N$, and each field except the address has constant length.
All deterministic rules and constraints of $X$ apart from the one that enforces $P_X$ can be expressed in $P_X$ using a constant number of bits, expressing $N$ and $e$ as variables that are initialized at the beginning of the program.
Then the running time of $M_U$ on $P_X$ is polylogarithmic in $e^N$.
The total length of $P_X$ is then $O(\log N + \log e)$.
For large enough $e$ and $N$, the machine $M_U$ is guaranteed to have enough time and space to complete its computation.

\begin{proposition}
  Let $\phi : \Gamma \to \Delta$ be the symbol map defined by $\phi(g) = \psi(S(g))$ for $g \in \Gamma \setminus \{0\}$ and $\phi(0)$ being the empty cube.
  Then $\phi(X) = Z$.
\end{proposition}

\begin{proof}
  By Lemma~\ref{lem:00SoficGluing}, $\phi(X)$ is $1$-zero-gluing, so it suffices to show that it has the same $1$-connected configurations as $Z$.
  Let $x \in X$ be a $1$-connected configuration.
  For each $r \geq 0$, let $C_r$ be the set of cells in $x|_{[-r,r]^3}$.
  Lemma~\ref{lem:MacrocellsCover} implies that if we ignore the z-coordinates of $C_r$, their simulated tiles form a pattern of $X_\tau$, so that their relative z-coordinates correspond to those given by the slopes of $X_\tau$.
  Hence $\phi(x)|_{[-r,r]^3}$ occurs in $Z$, and by compactness $\phi(x) \in Z$.
  
  Let then $z \in Z$ be $1$-connected.
  By definition, for all $r \geq 0$ the relative z-coordinates of all cells in $z|_{[-r,r]^3}$ come from the slopes of a pattern that occurs in $\tau^{n N}(s)$ for some $s \in \Sigma$ and $n \geq 0$.
  Let $K$ be an $X$-macrocell of level $n$ simulating $s$, and $x \in X$ a configuration containing $K$.
  Then the relative z-coordinates of cells in $C$ are given by the slopes of $\tau^{n N}(s)$, so $z|_{[-r,r]^3}$ occurs in $\phi(x)$.
  By continuity $z \in \phi(X)$.
\end{proof}

\subsection{The direction of determinism}

Finally, we apply a discrete shear map to $X$ in order to obtain an SFT that is deterministic in the direction of the z-axis, that is, the sparse subdynamics.
The north-determinism of $X$ comes from the fact that in a configuration $x \in X$, the symbol $x_{\vec 0}$ is determined by $x|_D$ where $D = [-1,1] \times \{-1\} \times [-1,1]$.
Define $f : \Z^3 \to \Z^3$ by $f(i,j,k) = (i,j,k+2j)$ and extend it to $\Gamma^{\Z^3}$ by $f(x)_{\vec v} = x_{f^{-1}(\vec v)}$.
Then the SFT $f(X)$ has the property that $x_{\vec 0}$ is determined by $x|_E$ where $E = f(D) = [-1,1] \times \{-1\} \times [-3,-1]$, so $f(X)$ is deterministic in the direction $(0,0,1)$.
It also has the same trace as $X$, implying the following result.

\begin{theorem}
  There exists a $\Z^3$-SFT whose $\Z$-projective subdynamics is $2$-sparse and that is deterministic in the direction of that subdynamics.
\end{theorem}

Since we can recode $f(X)$ into a spacetime subshift of a partial CA via a conjugacy whose neighborhood is contained in $\{0\}^2 \times \Z$, we have the following.

\begin{corollary}
  There exists a partial $\Z^2$-CA whose spacetime subshift $Y$ satisfies $\beta(Y) = 2$.
\end{corollary}

\section{Variants of the construction}
\label{sec:Variants}

\subsection{Wang cubes}
\label{sec:WangCubes}

We sketch a Wang cube version of $f(X)$, that is, a set of Wang cubes whose subshift $Y$ is sparse and deterministic.
As the cells of $f(X)$ are not connected to their neighbors in the $\ell_1$ metric of $\Z^3$, we represent each cell by a cluster of several Wang cubes.
Also, to obtain $\alpha(Y) = 4$ we perform some rescaling, which means that $Y$ is not conjugate to $f(X)$ (but shares its general structure).
The proof of $\alpha(Y) = 4$ is similar to that of Theorem~\ref{thm:ZIsSparse}.

To each cell $g \in \Gamma \setminus \{0\}$ of $f(X)$ we associate a cluster of Wang cubes whose shape depends on whether the north neighbor of $g$ is 1, 2 or 3 steps above $g$ in the z-direction (which can be determined from the contents of $g$).
The clusters all fit in cuboids of shape $2 \times 6 \times 4$.
They consist of a connected \emph{spine} of cells on the west half, and three \emph{wings} on the east half.
The \emph{base} of the cluster is at coordinate $(0,0,0)$ of the cuboid.
See Figure~\ref{fig:CubeCluster} for an illustration.

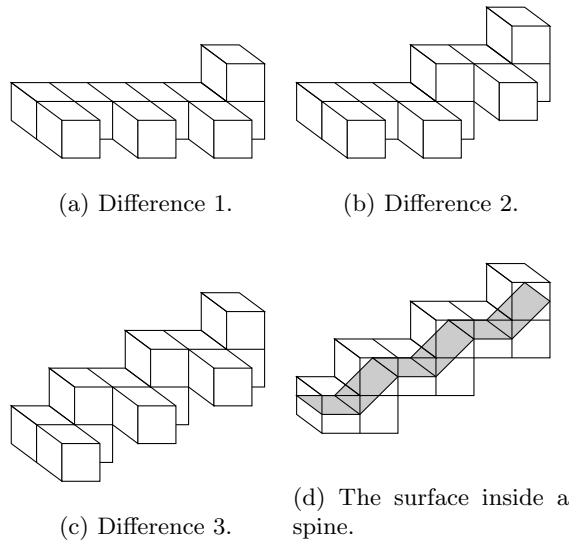
\begin{figure}[ht]
  \centering
  
  \begin{subfigure}{0.3\textwidth}
    \centering
    \begin{tikzpicture}[scale=0.5]
      \clip (-0.1,-4/3-0.1) rectangle (6.1+4/3,4.1);
      
      \foreach \x/\y in {0/0,5/1}{
        \draw (\x,\y) -- ++(0,1) -- ++(2/3,-0.5) -- ++(0,-1) -- cycle;
      }
      \draw (0,1) -- ++(5,0) -- ++(2/3,-0.5) -- ++(-5,0) -- cycle;
      \draw (5,2) -- ++(1,0) -- ++(2/3,-0.5) -- ++(-1,0) -- cycle;
      \foreach \x in {1,2,3,4,5}{
        \draw (\x,1) -- ++(2/3,-0.5);
      }
      \begin{scope}[shift={(2/3,-0.5)}]
        \draw (0,0) -- ++(6,0) -- ++(0,2);
        \foreach \x/\y in {0/0,2/0,4/0}{
          \draw [fill=white] (\x,\y) -- ++(0,1) -- ++(2/3,-0.5) -- ++(0,-1) -- cycle;
        }
        \foreach \x/\y in {1/1,3/1,5/1}{
          \draw (\x,\y) -- ++(2/3,-0.5);
        }
        \draw (5,1) -- ++(1,0);
      \end{scope}
      \begin{scope}[shift={(4/3,-1)}]
        \foreach \x/\y in {0/0,2/0,4/0}{
          \draw [fill=white] (\x,\y) -- ++(1,0) -- ++(0,1) -- ++(-1,0) -- cycle;
        }
      \end{scope}
    \end{tikzpicture}
    \caption{Difference 1.}
  \end{subfigure}
  \begin{subfigure}{0.3\textwidth}
    \centering
    \begin{tikzpicture}[scale=0.5]
      \clip (-0.1,-4/3-0.1) rectangle (6.1+4/3,4.1);
      
      \foreach \x/\y in {0/0,3/1,5/2}{
        \draw (\x,\y) -- ++(0,1) -- ++(2/3,-0.5) -- ++(0,-1) -- cycle;
      }
      \draw (0,1) -- ++(3,0) -- ++(2/3,-0.5) -- ++(-3,0) -- cycle;
      \draw (3,2) -- ++(2,0) -- ++(2/3,-0.5) -- ++(-2,0) -- cycle;
      \draw (5,3) -- ++(1,0) -- ++(2/3,-0.5) -- ++(-1,0) -- cycle;
      \foreach \x in {1,2,3}{
        \draw (\x,1) -- ++(2/3,-0.5);
      }
      \foreach \x in {4,5}{
        \draw (\x,2) -- ++(2/3,-0.5);
      }
      \begin{scope}[shift={(2/3,-0.5)}]
        \draw (0,0) -- ++(4,0) -- ++(0,1) -- ++(2,0) -- ++(0,2);
        \foreach \x/\y in {0/0,2/0,4/1}{
          \draw [fill=white] (\x,\y) -- ++(0,1) -- ++(2/3,-0.5) -- ++(0,-1) -- cycle;
        }
        \foreach \x/\y in {1/1,3/1,5/2}{
          \draw (\x,\y) -- ++(2/3,-0.5);
        }
        \draw (3,1) -- ++(1,0);
        \draw (5,2) -- ++(1,0);
      \end{scope}
      \begin{scope}[shift={(4/3,-1)}]
        \foreach \x/\y in {0/0,2/0,4/1}{
          \draw [fill=white] (\x,\y) -- ++(1,0) -- ++(0,1) -- ++(-1,0) -- cycle;
        }
      \end{scope}
    \end{tikzpicture}
    \caption{Difference 2.}
  \end{subfigure}
  
  \vspace{0.5cm}
  
  \begin{subfigure}{0.3\textwidth}
    \centering
    \begin{tikzpicture}[scale=0.5]
      \clip (-0.1,-4/3-0.1) rectangle (6.1+4/3,4.1);
      
      \foreach \x/\y in {0/0,1/1,3/2,5/3}{
        \draw (\x,\y) -- ++(0,1) -- ++(2/3,-0.5) -- ++(0,-1) -- cycle;
      }
      \draw (0,1) -- ++(1,0) -- ++(2/3,-0.5) -- ++(-1,0) -- cycle;
      \draw (1,2) -- ++(2,0) -- ++(2/3,-0.5) -- ++(-2,0) -- cycle;
      \draw (3,3) -- ++(2,0) -- ++(2/3,-0.5) -- ++(-2,0) -- cycle;
      \draw (5,4) -- ++(1,0) -- ++(2/3,-0.5) -- ++(-1,0) -- cycle;
      \draw (2,2) -- ++(2/3,-0.5);
      \draw (4,3) -- ++(2/3,-0.5);
      \begin{scope}[shift={(2/3,-0.5)}]
        \draw (0,0) -- ++(2,0) -- ++(0,1) -- ++(2,0) -- ++(0,1) -- ++(2,0) -- ++(0,2);
        \foreach \x/\y in {0/0,2/1,4/2}{
          \draw [fill=white] (\x,\y) -- ++(0,1) -- ++(2/3,-0.5) -- ++(0,-1) -- cycle;
        }
        \foreach \x/\y in {1/1,3/2,5/3}{
          \draw (\x,\y) -- ++(2/3,-0.5);
          \draw (\x,\y) -- ++(1,0);
        }
      \end{scope}
      \begin{scope}[shift={(4/3,-1)}]
        \foreach \x/\y in {0/0,2/1,4/2}{
          \draw [fill=white] (\x,\y) -- ++(1,0) -- ++(0,1) -- ++(-1,0) -- cycle;
        }
      \end{scope}
    \end{tikzpicture}
    \caption{Difference 3.}
  \end{subfigure}
  \begin{subfigure}{0.3\textwidth}
    \centering
    \begin{tikzpicture}[scale=0.5]
      \clip (-0.1,-4/3-0.1) rectangle (6.1+4/3,4.1);

      \foreach \x/\y in {0/0,2/1,4/2}{
        \draw [fill=black!20] (\x,\y+0.5) -- ++(1,0) -- ++(1,1) -- ++(2/3,-0.5) -- ++(-1,-1) -- ++(-1,0) -- cycle;
        \draw (\x+1,\y+0.5) -- ++(2/3,-0.5);
      }
      
      \foreach \x/\y in {0/0,1/1,3/2,5/3}{
        \draw (\x,\y) -- ++(0,1) -- ++(2/3,-0.5) -- ++(0,-1) -- cycle;
      }
      \draw (0,1) -- ++(1,0) -- ++(2/3,-0.5) -- ++(-1,0) -- cycle;
      \draw (1,2) -- ++(2,0) -- ++(2/3,-0.5) -- ++(-2,0) -- cycle;
      \draw (3,3) -- ++(2,0) -- ++(2/3,-0.5) -- ++(-2,0) -- cycle;
      \draw (5,4) -- ++(1,0) -- ++(2/3,-0.5) -- ++(-1,0) -- cycle;
      \draw (2,2) -- ++(2/3,-0.5);
      \draw (4,3) -- ++(2/3,-0.5);
      
      \begin{scope}[shift={(2/3,-0.5)}]
        \draw (0,0) -- ++(2,0) -- ++(0,1) -- ++(2,0) -- ++(0,1) -- ++(2,0) -- ++(0,2);
        \foreach \x/\y in {1/1,3/2,5/3}{
          \draw (\x,\y) -- ++(1,0);
          \draw (\x,\y) -- ++(0,-1);
        }
        \draw (2,2) -- ++(0,-1);
        \draw (4,3) -- ++(0,-1);
      \end{scope}
    \end{tikzpicture}
    \caption{The surface inside a spine.}
    \label{fig:SpineSurface}
  \end{subfigure}
  
  \caption{The three clusters and the surface inside a spine (shown for difference 3 along the z-axis). The perspective is as in Figure~\ref{fig:SlopeCubes}.}
  \label{fig:CubeCluster}
\end{figure}

Each cube in a cluster $C$ contains the contents of the cell $g$ and its relative position in the cluster, and they use their common faces to verify that this information is consistent within $C$.
The cluster connects to its north neighbor through the north face of the topmost cell, and to its other neighbors through the south, east and west faces of the base cube; if the base is at $(0,0,0)$, then north neighbor has its base at $(6,0,k)$ for some $k \in \{1,2,3\}$, the south neighbor at $(-6,0,k')$ for some $k' \in \{-1,-2,-3\}$, and the west and east neighbors at $(0,-1,0)$ and $(0,1,0)$.
Note that the spine of the east neighbor of $C$ will intersect one or more wings of $C$.
For this reason, we allow a Wang cube to be part of both the spine of a cluster and a wing of another, if those clusters are either neighbors or have north or south neighbors that form an east-west neighbor pair.
This condition can be verified locally by having the cubes store information about each cluster within distance 3 or less in the neighbor graph.
The rules for neighboring clusters are otherwise the same as in $f(X)$.
All faces that do not connect two cubes that are part of clusters have the color 0, and we add a blank cube that has this color on each of its faces.
Let $Y$ be the resulting SFT. The following proves Theorem~\ref{thm:Main}.

\begin{proposition}
  The SFT $Y$ is deterministic along the z-axis, and we have
  \[ \alpha(Y) = 4, \underline \alpha(Y) = 2, \beta(Y) = 2, \underline \beta(Y) = 2, \overline \beta(Y) = 4. \]
\end{proposition}

\begin{proof}[Proof sketch]
  Determinism of $Y$ follows from that of $f(X)$.
  
  Say that two non-blank Wang cubes are \emph{face-adjacent} if they share a face whose color is not 0, and \emph{face-connected} if they are connected by the face-adjacency relation.
  We first observe that in a configuration $y \in Y$ where all non-blank cubes are face-connected, each column has at most two non-blank cubes.
  Namely, in a configuration of $f(X)$ whose cells are connected by the neighbor relation, each column has at most one cell, so a column of $y$ cannot intersect two spines or two wings at different heights.
  It can intersect a spine and a wing, but since the y-coordinates of vertical pairs of cubes in the spines are disjoint from the y-coordinates of the wings, in this case it contains only one cube from each.

  We now prove $\alpha(Y) = 4$, and for that, let $y \in Y$.
  For every spine in $y$, consider a two-dimensional surface in $\R^3$ that runs through it in the north-south and east-west directions and splits the north-south faces in half, as in Figure~\ref{fig:SpineSurface}.
  Denote the union of these surfaces by $K$.
  If five non-blank cubes lie in the same column, then three of them, say $c_1, c_2, c_3$, lie in different face-connected components, and three points $\vec v_1, \vec v_2, \vec v_3 \in K$ on their west faces have equal horizontal projections.
  The clusters organize themselves into structures that correspond to $X$-macrocells of level $n$, so each $\vec v_i$ is part of arbitrarily large subsets $K_i^n$ of $K$ formed by surfaces of face-connected cubes that, when scaled appropriately, are within bounded distance of the mat $f'(T^3)$ in the mat metric, where $f'(i,j,k) = (i, 6j, k+j)$.
  By Lemma~\ref{lem:T3Lemma}, for large enough $n$ two of these subsets intersect, say $K_i^n$ and $K_j^n$.
  Thus some spine cubes whose surfaces are part of these sets are either equal or share a west-east face.
  A simple case analysis shows that in the latter case a wing of the west spine or its south or north neighbor intersects the east spine.
  Hence the cubes $c_i$ and $c_j$ are in fact face-connected, a contradiction.
  
 To see $\underline \alpha(Y), \beta(Y), \underline \beta(Y) \leq 2$, apply a conjugacy that joins pairs of vertically adjacent Wang tiles together. The lower bounds are obvious, by putting two surfaces on top of each other.
 
 The claim $\overline \beta(Y) = 4$ is less trivial, and we only outline it. Observe that whenever $X \cong Y$, the supports of configurations in $X$ are contained in sets $S+B_k$ where $S$ is the support of a configuration in $Y$ and $B_k$ is a ball of radius $k$. Thickening our surfaces this way can produce overlaps which are covered by two intervals of bounded length, but which can be arbitrarily far: when we thicken the edge of a bridge, it now hovers directly over the corresponding unbridge, and the height of the bridge of course can be arbitrary. However, it is easy to see that near such a discontinuity in our mats, there are bounds on the remaining vertical gaps, so indeed two intervals suffice to cover the traces arising from a single connected component.
 
In the proof of $\alpha(Y) = 4$ above, we showed that in a configuration of $Y$, any configuration can contain at most vertically overlapping connected components. It is easy to modify the analysis to obtain that the same remains true even if we blow up the components by $B_k$, so indeed $\overline \beta(Y) = 4$.
\end{proof}

\subsection{Binary alphabet}

We now construct a version of $f(X)$ that uses the binary alphabet and whose trace equals $X_{\leq 2}$, the one-dimensional sofic shift of configurations with at most two 1-symbols.
Say that a set $W \subset \{0,1\}^n$ of equal-length binary words is \emph{unbordered} if for all $v, w \in W$, no proper prefix of $v$ is a suffix of $w$.
Let $W \subset \{0,1\}^n$ be an unbordered set of size $|\Gamma|-1$, which can be found as a subset of $\{ 0^k w 1^k \;|\; w \in \{0,1\}^k \}$ for large enough $k$, and let $\gamma : W \to \Gamma \setminus \{0\}$ be a bijection.
We construct an SFT $Y \subset \{0,1\}^{\Z^3}$ as follows.
In a configuration $y \in Y$, if $y_{(i,j,k)} = 1$, then $w = y_{(i',j,k)} y_{(i'+1,j,k)} \cdots y_{(i'+n-1,j,k)} \in W$ for some $i' \in [i-n+1, i]$.
Because $W$ is unbordered, there is at most one such $i'$ for each $(i,j,k)$, and we say that the cell $g = \gamma(w)$ occurs in $y$ at $(i',j,k)$.
In this case a north neighbor of $g$ must occur at $(i'+n,j,k+p)$ for the $p \in \{1,2,3\}$ given by the simulated slope of $g$, and a south neighbor of $g$ must occur at $(i'-n,j,k-p')$ for some $p' \in \{1,2,3\}$.
If $g$ has an east neighbor, it must occur at $(i',j+1,k)$, and symmetrically for the west neighbor.
No other occurrences of elements of $\Gamma \setminus \{0\}$ may occur at $(i'+m,j+\ell,k+q)$ for any $m \in [-n,n]$, $\ell,q \in \{-1,0,1\}$.

Again with an argument similar to Theorem~\ref{thm:ZIsSparse} we can show that $Y$ is 2-sparse and deterministic along the z-axis. In fact, we see that no two cells that appear in coding words $w \in W$ (even with zero content) can appear in the same column.
We now modify $Y$ by changing from 0 to 1 every coordinate $y_{(i,j,k)}$ such that some element of $\Gamma \setminus \{0\}$ occurs at $y_{(i,j,k+1)}$, and call this new SFT $Y'$.
Since the first symbol of every word $w \in W$ is 0, $Y'$ is still 2-sparse, and since it is clearly conjugate to $Y$, it is deterministic.
Furthermore, the new 1-symbols allow the words $1 0^m 1$ for all $m \geq 0$ to occur in the trace of $Y'$, since we can place a bottom-most point of a down bridge of a level-$\infty$ macrocell over a top point of an up bridge of a level-$\infty$ macrocell on the same column at any distance.
We have shown the following.

\begin{proposition}
  There exists a binary $\Z^3$-SFT that is deterministic along the z-axis and whose projective subdynamics equals $X_{\leq 2}$.
\end{proposition}

\subsection{A simpler example}
\label{sec:SimplerExample}

In this section we give a much simpler substitution which can be used to obtain SFTs with sparse projective $\Z$-subdynamics, although with a worse bound on sparseness. The alphabet and the interpretation as a surface are the same as before, and we construct the zero-gluing subshift and its 0-to-0 sofic implementation as before.





\begin{definition}
The two-dimensional $3 \times 4$ substitution $\tau'$ over $\Sigma$ is defined by
\[ \tile{x} \mapsto \begin{smallmatrix}
\tile{x} & \tile{x} & \tile{x} \\
\rtile{ } & \lrtile{s} & \ltile{ } \\
\rtile{ } & \lrtile{n} & \ltile{ } \\
\tile{x} & \tile{x} & \tile{x} 
\end{smallmatrix}\]
for $\tile{x} \in \{\tile{ }, \tile{n}, \tile{s}\}$.
In the case of a tile with a jagged west border, the west border of the westmost tiles in the image is jagged, and symmetrically for east borders.
\end{definition}

\begin{proposition}
\label{prop:SimplerExample}
The three-dimensional SFT corresponding to $\tau'$ is $37$-sparse but not $3$-sparse.
\end{proposition}

\begin{proof}[Proof sketch]
  Suppose that we can find $x$ different cells in a single vertical column, and let $m$ be the diameter of this set.
  For each $n \geq 0$ the cells lie in different level-$n$ macrocells.
  Then by the pigeonhole principle, in $y = \lceil x/3 \rceil$ of them the slope of the macrocell is the same.
  In $z = \lceil y/3 \rceil$ of those, the y-coordinates of the cells lie in the same third of their macrocell (in terminology analogous to that of Section~\ref{sec:Mats}, in the left unbridge, the up bridge or the right unbridge).
  Now, if $h$ is the total height of the level-$n$ macrocell, divide the cells into four height classes depending on how high they are in their macrocell relative to its base (which can be defined arbitrarily, since we are restricting to macrocells of the same slope): $[0,h/4)$, $[h/4,h/2)$, $[h/2, 3h/4)$ and $[3h/4,h]$.
  In $w = \lceil z/3 \rceil$ of the cells, the height class is the same.
  As long as $m < h/4$, the condition $w \geq 2$ leads to a contradiction, showing that necessarily $x < 37$.

  It is easy to show that the subshift is not $3$-sparse, by putting the left and right unbridges of two $\infty$-tiles just on top of each other, and putting the tops of the up bridges of two $\infty$-tiles under them.
\end{proof}

We conjecture that the $\Z$-projective subdynamics of this example is indeed $4$-sparse.

\section{Open questions}

We are not able to prove the optimality of Theorem~\ref{thm:Main}.
The following proposition states that at least not all of the sparseness bounds can be dropped to $1$, which is the obvious lower bound for each of them for a nontrivial subshift.

\begin{proposition}
For any set of Wang cubes $C$ defining an SFT $X \subset C^{\Z^3}$ with $|X| \geq 2$, we have $\alpha(X) \geq 2$ and $\overline \beta(X) \geq 2$.
\end{proposition}

\begin{proof}
  Suppose $\overline \beta(X) = 1$.
  Then in every column of every $x \in X$, the maximum distance between non-blank cubes is bounded by some $m \in \N$, and in some $x$ there are non-blank cubes.
  Since the top face of the top non-blank cube and the bottom face of every bottom non-blank cube of each column has the color of the blank cube, we can place an additional copy of each non-blank cube $n$ steps above itself for any $n > m$.
  This implies $\overline \beta(X) \geq 2$, a contradiction.
  The other claim follows from $\alpha(X) \geq \overline \beta(X)$.
\end{proof}

\begin{question}
  Does there exist a set of Wang cubes $C$ such that the set of tilings $X \subset C^{\Z^3}$ satisfies one (or all) of $\alpha(X) = 2$, $\underline \alpha(X) = 1$, $\beta(X) = 1$, $\underline \beta(X) = 1$, or $\overline \beta(X) = 2$?
\end{question}

We also do not know whether $\beta(X) = 2$ is optimal among general zero-gluing $\Z^3$-subshifts.
In terms of concrete projective subdynamics, a natural question is the following.

\begin{question}
  Does there exist a three-dimensional subshift of finite type $X \subset \{0,1\}^{\Z^3}$ whose $\Z$-projective subdynamics is $X_{\leq 1}$?
\end{question}

More generally, one can ask for a full classification of possible sofic subdynamics. All one-dimensional sofic shifts without a \emph{universal period} were implemented with $\Z^2$-SFTs in \cite{PaSc15} (see the article for the definition of the universal period), so they are $\Z$-projective subdynamics of $\Z^d$-SFTs for all $d \geq 2$.

\begin{question}
  Which $\Z$-sofics with a universal period are $\Z$-projective subdynamics of $\Z^d$-SFTs, for each $d \geq 3$?
\end{question}

Even more generally one may ask for a classification of all projective subdynamics; these have not been classified even in dimension two, but in principle the $\Z$-traces of $\Z^d$-shifts could be a larger but simpler class. In \cite{PaSc15} it is shown that some one-dimensional computable subshifts are not traces of $\Z^d$-SFTs for any $d$. It is a particularly interesting question whether the set of possible traces is the same for all $d \geq 3$. 

Let $G$ be a finitely generated group.
Then a zero-gluing $G \times \Z$-subshift has a well-defined $\Z$-trace and we can discuss its sparseness.
If the subshift is deterministic in the direction of the trace, it can be seen as the spacetime subshift of a cellular automaton on a $G$-subshift.
This is exactly the setting of~\cite{SaTo21}, where it was shown that for a large class of groups $G$ and subshifts on them, no such CA with sparse traces exist.
The case of the free groups was explicitly left open, and we are not able to construct even a non-deterministic example.

\begin{question}
  Let $F_2$ be the free group on two generators.
  Does there exist a $F_2 \times \Z$-SFT with sparse $\Z$-trace, or even such a zero-gluing $F_2 \times \Z$-subshift?
\end{question}

We can also generalize the subgroup defining the trace subshift.

\begin{question}
  For which pairs of finitely generated groups $H \leq G$ does there exist a $G$-SFT or a zero-gluing $G$-subshift with sparse $H$-trace?
\end{question}

While we have made our SFT vertically deterministic in one direction, we had technical difficulties making the other direction deterministic, and leave this question open.

\begin{question}
  Does there exist a three-dimensional subshift of finite type $X \subset \{0,1\}^{\Z^3}$ whose $\Z$-projective subdynamics is sparse, such that the orthogonal $\Z^2$-subdynamics is expansive? 
\end{question}

Let $\mathcal{S} \subset \N$ be the set of Gödel numbers of $\Z^3$-SFTs that have sparse $\Z$-traces. Usually natural sets like these are complete for some early level of the arithmetical hierarchy. It is easy to show that $\mathcal{S}$ is undecidable, and is at most $\Sigma^0_3$.

\begin{question}
Where is $\mathcal{S}$ in the arithmetical hierarchy?
\end{question}


\bibliographystyle{plain}
\bibliography{../../../bib/bib}{}

\end{document}